\documentclass[11pt]{article}
\usepackage{amsfonts}
\usepackage{amsmath}
\usepackage{amssymb}
\usepackage{mathrsfs}
\usepackage{graphicx}
\usepackage{amsbsy}
\usepackage{theorem}
\usepackage{color}
\usepackage[colorlinks=true]{hyperref}
\usepackage{tikz}
\usepackage{bbm}
\usepackage[normalem]{ulem}
\newtheorem{theorem}{Theorem}[section]
\newtheorem{lemma}[theorem]{Lemma}

\newtheorem{definition}{Definition}
\def\thetheorem{\thesection.\arabic{theorem}}
\def\thesection{\arabic{section}}

\def\beq{\begin{equation}\displaystyle}
\def\eeq{\end{equation}}
\def\bel{\begin{equation} \displaystyle \begin{array}{l} }
\def\eel{\end{array} \end{equation} }
\def\bell{\begin{equation} \displaystyle \begin{array}{ll}  }
\def\eell{\end{array} \end{equation} }

\def\bea{\begin{eqnarray}}
\def\eea{\end{eqnarray} }
\def\bean{\begin{eqnarray*}}
\def\eean{\end{eqnarray*} }
\newenvironment{proof}{\noindent{\bf Proof.~}}{{\mbox{}\hfill {\small \fbox{}}\\}}
\newcommand{\comment}[1]{}

\catcode`@=11
\renewcommand\appendix{\bigskip {\noindent \Large \bf Appendix}
  \setcounter{section}{0}%
  \setcounter{subsection}{0}%
\setcounter{equation}{0}%
\setcounter{theorem}{0}%
\def\thesection{A.\arabic{section}}
\def\thetheorem{A.\arabic{theorem}}
\def\theequation {A.\arabic{equation}}}
\catcode`@=12

\def\NN{\mathbb{N}}

\def\RR{\mathbb{R}}

\def\ds{\displaystyle}

\def\bar#1{{\overline #1}}

\def\calR{{\mathcal R}}

\definecolor{mygreen}{rgb}{0,0.65,0}

\def\ess{\mathrm{ess}}

\begin{document}


%

\title{Final size and convergence rate for an epidemic in \\ heterogeneous population}

\author{ Luis Almeida\thanks{Sorbonne Universit{\'e}, CNRS, Universit\'{e} de Paris, Inria, Laboratoire Jacques-Louis Lions UMR7598, F-75005 Paris, France. Email : \texttt{\{almeida,gregoire.nadin,benoit.perthame\}@ljll.math.upmc.fr}}
\and
Pierre-Alexandre Bliman\thanks{Sorbonne Universit{\'e}, Inria, CNRS, Universit\'{e} de Paris, Laboratoire Jacques-Louis Lions UMR7598, Equipe MAMBA, F-75005 Paris, France. Email : \texttt{pierre-alexandre.bliman@inria.fr}}
\and
Gr\'egoire Nadin\footnotemark[1]
\and
Beno\^ \i t Perthame\footnotemark[1]
\thanks{B.P. has received funding from the European Research Council (ERC) under the European Union's Horizon 2020 research and innovation programme (grant agreement No 740623). }
\and Nicolas Vauchelet\thanks{LAGA, UMR 7539, CNRS, Universit\'e Sorbonne Paris Nord,99 avenue Jean-Baptiste Cl\'ement, 93430 Villetaneuse, France.
  Email : \texttt{vauchelet@math.univ-paris13.fr}}
}

\maketitle

\begin{abstract}
We formulate a general SEIR epidemic model in a heterogenous population characterized by some trait in a discrete or continuous subset of a space $\RR^d$.
The incubation and recovery rates governing the evolution of each homogenous subpopulation depend upon this trait, and no restriction is assumed on the contact matrix that defines the probability for an individual of a given trait to be infected by an individual with another trait.
Our goal is to derive and study the final size equation fulfilled by the limit distribution of the population.
We show that this limit exists and satisfies the final size equation.
The main contribution is to prove the uniqueness of this solution among the distributions smaller than the initial condition.
We also establish that the dominant eigenvalue of the next-generation operator (whose initial value is equal to the basic reproduction number) decreases along every trajectory until a limit smaller than 1.
The results are shown to remain valid in presence of diffusion term.
They generalize previous works corresponding to finite number of traits (including  metapopulation models) or to rank 1 contact matrix (modeling e.g.\ susceptibility or infectivity presenting heterogeneity independently of one another).
\end{abstract}



\begin{color}{black}

\section{Introduction}

Observing the complex behaviour of the COVID-19, see Refs.~\cite{Gomes2,Britton,Britton2,DiLauro,Dolbeault,Gomes}, reveals the importance of the population heterogeneity in the dynamics of an outbreak.
Inspired by this question, but in contrast with recent papers aiming at reproducing the dynamics of the pandemic observed through various data, we adopt here an abstract point of view and investigate issues related to the concepts of epidemic final size and herd immunity in an ample setting.

Before explaining in more details the context and contribution, we present the system under study.
We consider here the general SEIR epidemic model \eqref{eqG1}, which describes the spread of a disease in a heterogeneous population characterized by a ``trait" $x\in \Omega$, for a given open subset $\Omega \subset \RR^d$.
The underlying structure described by this trait can be very varied, typical examples being one or several of the following characteristics: propensity to have social contacts and hazardousness of the latter (in the sense of the transmission disease), susceptibility, infectivity, characteristics of the immunological response, age, spatial location, etc.
\begin{subequations}
\label{eqG1}
\begin{gather}
\partial_t S(t,x) = -S(t,x) \int_\Omega \beta(x,y) I(t,y)\ dy,\qquad  S(0,x) = S_0(x) \label{eqG1a}\\
\partial_t E(t,x) =  S(t,x) \int_\Omega  \beta(x,y) I(t,y)\ dy - \alpha(x) E(t,x),\qquad E(0,x) = E_0(x),\label{eqG1c}\\
\partial_t I(t,x) =  \alpha (x) E(t,x) - \gamma(x) I(t,x),\qquad I(0,x) = I_0(x), \label{eqG1b}  \\
\partial_t R(t,x) = \gamma(x) I(t,x), \qquad R(0,x) = R_0(x). \label{eqG1d}
\end{gather}
\end{subequations}
Classical notations are used~: $S$ denotes the proportion of susceptible individuals, $E$ the proportion of those who have been exposed to the disease, $I$ the proportion of infected individuals, $R$ the proportion of removed individuals.
The component $\beta(x,y)$ of the so-called (infinite-dimensional) {\em contact matrix} $\beta$ is the product of the average contact rate between individuals of traits $x$ and $y$ and average transmission probability from an individual with trait $y$ to individual with trait $x$.
The force of infection $\int_\Omega \beta(x,y) I(t,y)\ dy$ that applies to the individuals of trait $x$ thus obeys the law of mass-action.
The average incubation rate is denoted $\alpha(x)$ and the average recovery rate is denoted $\gamma(x)$:
both quantities may depend upon the trait $x$.

As can be seen, the dynamics of the outbreak is assumed to be sufficiently fast to neglect the effects of the demography in the population.
A particular case of system \eqref{eqG1} is when the parameters $\alpha,\beta,\gamma$ are independent of the trait: the classical (``homogeneous") SEIR model, which reads
\begin{equation}
\label{eq454}
\dot S = -\beta SI,\qquad
\dot E = \beta SI-\alpha E,\qquad
\dot I = \alpha E -\gamma I,\qquad
 \dot R = \gamma I
\end{equation}
is then recovered after integration over $\Omega$.
A variant of \eqref{eqG1} with diffusion operator in the infected compartment capable to account for possible mutations or spatial displacement of the infected population is also considered, see equation \eqref{eqGLaplacian} below.

For an epidemic model where the effects of demography are neglected, the number of exposed and infected individuals is expected to go to zero after reaching an epidemic peak (if the initial proportion of susceptible is sufficient).
An important question is then to determine what will be the so-called {\em epidemic final size}, that is the cumulative number (or proportion) of individuals infected, and thus ultimately removed, during the outbreak.
For the homogeneous SEIR model \eqref{eq454}, the answer to this question is well-known. Indeed, due to the invariance of the function $S(t)+E(t)+I(t)-\dfrac{\gamma}{\beta}\ln S(t)$ along the trajectory, the asymptotic value $S_\infty$ of $S(t)$  is necessarily a root of the equation $S_\infty-\dfrac{\gamma}{\beta}\ln S_\infty=S_0+E_0+I_0-\dfrac{\gamma}{\beta}\ln S_0$.
Indeed, one may show that it is the unique root of this equation smaller than the so-called {\em herd immunity threshold} $\dfrac{\gamma}{\beta}$ ---the value of $S$ below which the number $E+I$ of exposed and infected necessarily decreases.

The study of the final size of an epidemic has been addressed in several works since its appearance in early contributions\cite{Bailey,Kermack}.
Ref.~\cite{Ma} was among the first papers to unveil the generality of this concept, extending it to models with several distinct infectious stages, or with arbitrarily distributed mean contact rate, or again having spatially heterogeneous contact structures.
Multiple susceptible classes were considered in Ref.~\cite{Arino}.
Ref.~\cite{Brauer} established final size equation for age-of-infection model.
Distributed susceptibility or distributed infectivity has been studied in Ref.~\cite{Novozhilov}, and Ref.~\cite{Andreasen} studied an epidemic model within a population of $n$ different groups.
Analysis in a general framework was presented in Ref.~\cite{Miller}, establishing sufficient conditions under which epidemic final size equation exists.
Sharp estimates and bounds are obtained in Ref.~\cite{Katriel} for general model with heterogeneous susceptibility.
Also, the case of a two-group SIR model has been studied in Ref.~\cite{Magal}. 
Finally, we mention that, in a different biological context, related questions have been addressed on models of population dynamics with heterogeneity modeled by an internal variables\cite{Desvillettes,Jabin}.

An important quantity to describe the dynamics of an epidemic is the {\em basic reproduction number} $\calR_0$. From a biological point of view, the latter represents the expected number of secondary cases directly generated by an individual in a completely susceptible population during its entire period of infectiousness.
A famous threshold property states that the disease can invade (in the sense that it will spread and reach an epidemic peak) if $\calR_0>1$ and the initial number of susceptible is sufficient, whereas it cannot if $\calR_0<1$.
Mathematically, this number has been defined in full generality in the seminal paper\cite{DiekmannR0} as the dominant eigenvalue of the so-called {\em next-generation operator}.

Our aim in the present paper is to extend some of these previous results and completely characterize  the final size of an epidemic in a heterogeneous population whose dynamics is governed by the general systems \eqref{eqG1} or \eqref{eqGLaplacian}.
On this occasion, we revisit and present in a unified setting the notions of herd immunity and basic reproductive number, and demonstrate the centrality and powerfulness of the concept of next-generation operator.

The outline of this work is the following.
We present in Section \ref{sec:AMR} the assumptions used in the sequel, and provide the key results, Theorems \ref{th1} and \ref{th2}.
They correspond respectively to epidemic spreading without and with diffusion.
Their proof is given in Sections \ref{sec:gen} and \ref{sec:genDiff}.
Section \ref{se5} shows how the present framework includes several examples previously considered in the literature.

\section{Assumptions and main results}
\label{sec:AMR} 


\paragraph{Heterogeneous populations without diffusion.}
To begin with, we investigate the final size of an epidemic in a heterogeneous population whose dynamics is modelled by system~\eqref{eqG1}.
Denote $L^\infty_+(\Omega)$ the set of measurable functions which are bounded and nonnegative a.e.\ on $\Omega$, and assume that
\begin{equation}  \label{eq:hyp}
\left\{ \begin{array}{l}
\displaystyle
\alpha \in L^\infty_+(\Omega)  \ \text{ and } \  \underline \alpha :=  \ess\inf_{x\in \Omega} \alpha(x) >0,  \\[5pt]
\displaystyle
\beta\in L^\infty_+(\Omega\times \Omega) \cap 
  L^2(\Omega\times \Omega)  \  \text{ and } \ \beta >0,
\\[5pt]  
\displaystyle
\gamma\in L^\infty_+(\Omega) \ \text{ and } \ \underline \gamma:= \ess\inf_{x\in \Omega} \gamma(x)>0.
\end{array} \right.  
\end{equation}
In order for the disease to start, we assume also that the initial proportions of susceptible individuals, and of infected and exposed individuals, are nonzero~:
\begin{equation}
  \label{eq:init}
  S_0, E_0, I_0, R_0 \in L^\infty_+(\Omega), \quad   S_0>0,
  \quad E_0+ I_0 \not\equiv 0, \quad
  \int_\Omega (S_0+E_0+I_0+R_0)\,dx = 1.
\end{equation}
The normalization assumption contained in the last identity materializes the fact that the variables constitute proportions of a given population whose overall value is conserved along time.

The first result concerns system~\eqref{eqG1}.
\begin{theorem}\label{th1}
  Assume that assumptions \eqref{eq:hyp} and \eqref{eq:init} hold, then for the solution of problem \eqref{eqG1},
  $t\mapsto S(t,\cdot)$ is a decreasing function and there exist $S_\infty, R_\infty \in L^\infty_+(\Omega)\cap L^1(\Omega)$ such that, for a.e.\ $x\in \Omega$,
\begin{subequations}
  \label{eq18}
  \begin{align}
    & \lim_{t\to +\infty} S(t,x) = S_\infty(x), \quad
    \lim_{t\to +\infty} R(t,x) = R_\infty(x), \\
    & \lim_{t\to +\infty} I(t,x) = \lim_{t\to +\infty} E(t,x) = 0.
  \end{align}
\end{subequations}
In addition, the convergence is uniform, $S_\infty>0$, $S_\infty$ is the unique solution such that $S_\infty \leq S_0$ of the {\em final size equation}
  \begin{equation}\label{eq:Sinf}
    \ln S_\infty(x) - \int_\Omega \frac{\beta(x,y)}{\gamma(y)} S_\infty(y)\,dy
    = \ln S_0(x) - \int_\Omega \frac{\beta(x,y)}{\gamma(y)} \big( S_0(y) + E_0(y) + I_0(y)\big)\,dy,
  \end{equation}

and $R_\infty = S_0+E_0+I_0+R_0-S_\infty$.

Moreover, if
 \begin{equation}
 \label{eq888}
 \ess\inf_\Omega S_0>0\qquad \text{and}\qquad
 \underline{\beta}:=  \ess\inf_{(x,y)\in \Omega\times \Omega} \beta(x,y) >0,
 \end{equation}
 then $\ess\inf_\Omega S_\infty>0$ and the convergence in \eqref{eq18} is exponential. 
\end{theorem}

An alternative form of \eqref{eq:Sinf}, useful in the sequel, may be expressed as follows:
\begin{equation}
\label{eq45}
S_\infty(x) = S_0(x) \exp\left(
\int_\Omega \frac{\beta(x,y)}{\gamma(y)} \left(
S_\infty(y) - (S_0(y) + E_0(y) + I_0(y))
\right)\,dy
\right).
  \end{equation}

The proof of Theorem \ref{th1} is the subject of Section \ref{sec:gen}.
The main difficulty in this demonstration is to prove uniqueness of the solution of the final size equation \eqref{eq:Sinf}, in the set of those $S$ at most equal to the initial condition $S_0$.
It uses as central tool an irreducibility property related to the assumption $\beta >0$.
It is possible to relax this hypothesis by only assuming that the next-generation operator (defined in \eqref{def:Kt} below) is strongly positive ---this is indeed the case if $\beta>0$.
On the contrary, $\beta\geq 0$ is not sufficient to deduce uniqueness, see a counterexample at the end of the present Section.

\paragraph{Heterogeneous populations with diffusion.}
Diffusion is introduced in the infected compartment for the second result, allowing to consider possible additional natural phenomenon as mutations or spatial diffusion of infected individuals. There is a long history on this problem, in particular to study the effect of diffusion on the basic reproduction number, the spread or extinction of the epidemic and the existence of propagating waves, see Refs.~\cite{RD1,RD2,Kendall57,Thieme77} and the references therein. Here, we are interested in characterizing the final size of the epidemic in a bounded domain.

The  corresponding system reads, for $t \geq 0$, $x\in \Omega$, 
\begin{subequations}
\label{eqGLaplacian}
\begin{gather}
\partial_t S(t,x) = -S(t,x) \int_\Omega \beta(x,y) I(t,y)\ dy,\qquad S(0,x) = S_0(x), \label{eqGLaplacianS} \\
\partial_t E(t,x) =  S(t,x) \int_\Omega \beta(x,y) I(t,y)\ dy - \alpha(x) E(t,x),\qquad E(0,x) = E_0(x),  \label{eqGLaplacianE} \\
\left\{ \begin{array}{l}
\partial_t I(t,x) =  \alpha (x) E(t,x) - \gamma(x) I(t,x)+\Delta I (t,x),\qquad I(0,x) = I_0(x), \label{eqGLaplacianI} \\[5pt]
\partial_{n}I(t,x)=0 \quad \hbox{ on  } \quad  (0,\infty)\times \partial\Omega, 
\end{array} \right.
\\
\partial_t R(t,x) = \gamma(x) I(t,x), \qquad R(0,x) = R_0(x). \label{eqGLaplacianR}
\end{gather}
\end{subequations}

Essentially the same assumptions than Theorem~\ref{th1} allow to obtain the similar result.
\begin{theorem}\label{th2}
Assume that \eqref{eq:hyp} and \eqref{eq:init} hold. Assume moreover that $\Omega$ is a smooth, bounded open  set.
  Then, for the solution of \eqref{eqGLaplacian}, $t\mapsto S(t,\cdot)$ is a  decreasing function and there exist  $S_\infty,R_\infty\in L^\infty_+(\Omega)\cap L^1(\Omega)$ such that \eqref{eq18} holds for a.e.\ $x\in \Omega$, 
uniformly.
    
Moreover, denoting $\Phi_0$ the solution of
 \begin{equation} \label{eq:Phi_0}
   \left\{ \begin{array}{l}
    -\Delta \Phi_0 + \gamma \Phi_0 = S_0 + E_0 + I_0,
    \\[5pt]
     \partial_{n}\Phi_{0}=0 \quad  \hbox{ on }  \quad  (0,\infty)\times \partial\Omega,
\end{array} \right.
\end{equation}

$S_\infty$
is characterized by
\begin{subequations}
\label{eq46}
\begin{equation}
\label{eq46a}
S_\infty = S_0 \exp \left(\int_\Omega \beta(x,y)(\Phi_\infty(y)-\Phi_0(y))\,dy\right),
\end{equation}
where $\Phi_\infty$ is the unique solution such that $\Phi_\infty \leq \Phi_0$ (with $\Phi_0$ defined in \eqref{eq:Phi_0}) of
  \begin{equation} \label{eq:Phi_inf}
  \left\{ \begin{array}{l}
    -\Delta \Phi_\infty + \gamma(x) \Phi_\infty = S_0 \exp \left(\int_\Omega \beta(x,y)(\Phi_\infty(y)-\Phi_0(y))\,dy\right),
    \\[5pt]
     \partial_{n}\Phi_{\infty}=0 \quad \hbox{ on } \quad (0,\infty)\times \partial\Omega.
     \end{array} \right.
  \end{equation}
\end{subequations}

Last, if \eqref{eq888} holds, then the convergence is exponential.
\end{theorem}

Equation~\eqref{eq46} constitutes, together with \eqref{eq:Phi_0}, the final size equation for equation \eqref{eqGLaplacian}.
It appears analogous to~\eqref{eq:Sinf} when ignoring the diffusion term.
It may also be written as 
\[
  -\Delta \Phi_\infty + \gamma \Phi_\infty = S_\infty.
\]

Again, we can relax the hypothesis $\beta>0$ by only assuming that the modified next-generation operator (defined in \eqref{eq:Kdelta} below) is strongly positive, which is of course true when $\beta>0$. Another possible set of assumption guaranteeing this strong positivity, and thus the availability of Theorem \ref{th2},
is $\beta\geq 0$ and $\Omega$ connected. In that case, even if $\beta=0$ in some parts of $\Omega$, the diffusion ensures that the infection reaches such parts. If $\beta$ is not positive and $\Omega$ is not connected, it is easy to construct counter-examples where a connected part of $\Omega$ is invaded whereas another is not.

Finally, we mention that these results may be adapted straightforwardly when the exposed compartment is neglected, that is for the heterogeneous SIR model.
Also they may be extended to other operators than the diffusion, as the fractional Laplacian or integral operators, as long as a standard spectral theory is available. Last, diffusion could be included on the Exposed compartment with the same results and methods. 

\paragraph{Nonuniqueness and positivity of $\beta$}
    In order to illustrate the necessity of assumption $\beta>0$ in \eqref{eq:hyp}, we provide a simple example for the discrete case with piecewise constant coefficient. More precisely, assume that for some disjoint sets $\Omega_1, \Omega_2$ for which $\Omega = \Omega_1\cup\Omega_2$, and some positive constants $\alpha_k, \beta_k, \gamma_k$, $k=1,2$ and $\beta_{12}$,
 the coefficients verify:
\begin{gather*}
 \alpha(x) = \alpha_1 \mathbf{1}_{\Omega_1}(x)+\alpha_2 \mathbf{1}_{\Omega_2}(x),\qquad
 \gamma(x) = \gamma_1 \mathbf{1}_{\Omega_1}(x) + \gamma_2 \mathbf{1}_{\Omega_2}(x),\\
       \beta(x,y) = \left\{
        \begin{array}{ll}
          \beta_k\,, \quad &\text{if } (x,y)\in\Omega_k\times\Omega_k,\ k=1,2   \\
          \beta_{12}\,, \quad &\text{if } (x,y)\in\Omega_1\times\Omega_2   \\
          0\,, \quad &\text{if } (x,y)\in\Omega_2\times\Omega_1.
        \end{array}\right.
\end{gather*}

    Then, denoting
$$
S_k(t) := 
\int_{\Omega_k} S(t,x)\,dx,\
E_k(t) := 
\int_{\Omega_k} E(t,x)\,dx,\
I_k(t) := 
\int_{\Omega_k} I(t,x)\,dx
$$
for $k=1,2$, system \eqref{eqG1} implies (after integration in space) that
    \begin{align*}
      \dot S_1 = - S_1 (\beta_1 I_1 + \beta_{12} I_2), \qquad
      &\dot S_2 = - \beta_2 S_2 I_2, \\
      \dot E_1 = S_1 (\beta_1 I_1 + \beta_{12} I_2) - \alpha_1 E_1, \qquad
      &\dot E_2 = \beta_2 S_2 I_2 - \alpha_2 E_2, \\
      \dot I_1 = \alpha_1 E_1 - \gamma_1 I_1, \qquad
      &\dot I_2 = \alpha_2 E_2 - \gamma_2 I_2,
    \end{align*}
    complemented with initial data $S_k^0$, $E_k^0$, $I_k^0$, for $k=1,2$ (with the notations $S_k^0 =
    \int_{\Omega_k} S_0(x)\,dx$, and so on).
Denoting $S_{k\infty}$ the limit of $S_k(t)$ when $t\to +\infty$, $k=1,2$, the final state is characterized, in this particular setting, by the system of two scalar identities
    \begin{subequations}
    \label{eq55}
    \begin{align}
      & \ln S_{1\infty} - \frac{\beta_1}{\gamma_1} S_{1\infty} - \frac{\beta_{12}}{\gamma_2} S_{2\infty} = \ln S_{1}^0 - \frac{\beta_1}{\gamma_1} (S_{1}^0+E_1^0+I_1^0) - \frac{\beta_{12}}{\gamma_2} (S_{2}^0+E_2^0+I_2^0),
      \label{eqS1inf} \\
      & \ln S_{2\infty} - \frac{\beta_{2}}{\gamma_2} S_{2\infty} = \ln S_{2}^0 - \frac{\beta_{2}}{\gamma_2} (S_{2}^0+E_2^0+I_2^0).
        \label{eqS2inf}
    \end{align}
    \end{subequations}
    If $E_2^0+I_2^0 = 0$, then \eqref{eqS2inf} admits two solutions, namely $S_2^0$ (corresponding to absence of disease in $\Omega_2$), and another solution, denoted $\underline{S_2}$, which is strictly smaller than $S_2^0$ (corresponding to a spread of the epidemic in $\Omega_2$).
    If now $E^0_1+I_1^0 > 0$, for each of these solutions, solving \eqref{eqS1inf} for $S_{1\infty}\leq S_1^0$ yields a unique solution.
    We thus obtain two distinct solutions of equation \eqref{eq55}, smaller or equal to $(S^1_0,S^2_0)$.
    Notice that, for any initial condition of \eqref{eq:Sinf} constant on each of the two sets $\Omega_1, \Omega_2$, the solution itself retains this property.
    The previous considerations thus allows to exhibit for such a choice of initial condition, two distinct solutions of the final size equation \eqref{eq:Sinf}.
As a conclusion, the uniqueness property stated in Theorem \ref{th1} may not hold if $\beta\not > 0$.

\section{General contact matrices in absence of diffusion}
\label{sec:gen}

We investigate here system \eqref{eqG1}.
After introducing some adequate notions and results, Theorem~\ref{th1} is proved.
Some conserved quantities are first studied in Section \ref{se31}.
The notions of {\em basic reproduction number} and {\em herd immunity domain} are then introduced in this general setting in Section \ref{se32}, based on the next-generation operator.
The limit behaviour is established in Section \ref{se33}, together with the properties of the solution of the final size equation, formally achieving the proof of Theorem~\ref{th1}.
Section \ref{se34} then exploits the tools introduced during this section to disentangle the link between herd immunity and equilibrium stability.

\subsection{Conserved quantities and long time limit}
\label{se31}

A central role in the dynamics is played by the integral quantity
\begin{equation}  \label{eq11}
    \Phi (t,x)= \int_\Omega \frac{\beta (x,y)}{\gamma(y)} [S(t,y)+E(t,y)+I(t,y)]\,dy,
\end{equation}
extending a similar and standard quantity for the SIR model.

\begin{lemma}\label{lem:conserv2}
  Under assumption \eqref{eq:hyp}, the solution of \eqref{eqG1} verifies that, for a.e.\ $x \in  \Omega$, the quantities 
  $S(t,x)+E(t,x)+I(t,x)+R(t,x)$ and $\ln S(t,x) - \Phi(t,x)$ are constant with respect to time.
\end{lemma}

\begin{proof}
  Firstly, adding the equations in \eqref{eqG1} provides directly the conservation of $S+E+I+R$ with respect to time. Secondly, we compute
  \begin{align*}
    \frac{\partial}{\partial t} [  \ln S(t,x) -\Phi (t,x)  ]\ 
    & = \frac{\dot S(t,x)}{S(t,x)} - \int_\Omega \frac{\beta (x,y)}{\gamma(y)} (\dot S(t,y) + \dot E(t,y) + \dot I(t,y))\,dy  \\
    & = - \int_\Omega \beta(x,y) I(t,y)\,dy\ + \int_\Omega \beta (x,y) I(t,y) =0.
  \end{align*}
We thus deduce the conservation relation
  \begin{equation} \label{Lyap3}
    \frac{\partial}{\partial t} [  \ln S(t,x) -\Phi (t,x)  ]=0,
  \end{equation}
 and Lemma~\ref{lem:conserv2} is proved.
\end{proof}

We now demonstrate the basic convergence properties stated in Theorem \ref{th1}.
 The exponential convergence, which necessitates more advanced tools, will be treated later, see Section~\ref{se33}.
 
\begin{lemma}\label{lem:lim}
  Assume that \eqref{eq:hyp} and \eqref{eq:init} hold.
  Then the solution of \eqref{eqG1} verifies~: there exists $S_\infty \in L^\infty_+(\Omega)$, $0 < S_\infty\leq S_0$, \textcolor{magenta}{and $R_\infty = S_0+E_0+I_0+R_0-S_\infty$}, such that the convergence properties stated in \eqref{eq18}  hold uniformly, for a.e.\ $x\in \Omega$, as well as \eqref{eq:Sinf}. 
\end{lemma}

\begin{proof}
  We deduce easily from \eqref{eqG1} that for a.e.\ $x\in\Omega$, $t\mapsto S(t,x)$ is decreasing.
  Being bounded from below by 0, it admits a limit as $t$ goes to $+\infty$.
  By the same token, $S(\cdot,x)+E(\cdot,x)$ and $S(\cdot,x)+E(\cdot,x)+I(\cdot,x)$ admit similarly a limit.
  As $S(\cdot,x)$ admits a limit, we deduce that {\color{black} $E(\cdot,x) = (S(\cdot,x)+E(\cdot,x)) -S(\cdot,x)$ and $I(\cdot,x) = (S(\cdot,x)+E(\cdot,x)+I(\cdot,x)) - (S(\cdot,x)+E(\cdot,x))$ converge as well} when $t\to +\infty$.
{\color{black} The convergence of $R$ then comes as a consequence of the first conservation property in Lemma \ref{lem:conserv2}. 
 
 Equation \eqref{eq:Sinf} and its equivalent form \eqref{eq45} follow from Lemma \ref{lem:conserv2}, allowing to deduce that $S_\infty>0$.
From the fact that $\displaystyle\lim_{t\to\infty} S(t,x)=S_\infty(x)$ for a.e.\ $x \in\Omega$, one deduces, by the dominated convergence theorem, that this convergence is indeed $L^1(\Omega)$. 
Then it follows from equation \eqref{eqG1a} by application of Arzela-Ascoli's theorem that the convergence of $S(t,x)$ towards $S_\infty(x)$ is uniform a.e.~over $\Omega$.

Let us now study the limits of $E$ and $I$.
Adding \eqref{eqG1a} and \eqref{eqG1b} yields
$$
\partial_t (S(t,x)+E(t,x)) +\alpha(x) (S(t,x)+E(t,x)) = \alpha(x) S(t,x),
$$
with $S(t,x) \to S_\infty(x)$ a.e.~in $\Omega$ uniformly.
Due to the fact that $\alpha(x) \geq \underline\alpha >0$, see \eqref{eq:hyp}, one obtains by integration that $S(t,x)+E(t,x) \to S_\infty(x)$ when $t\to +\infty$, a.e.~and uniformly on $\Omega$.
Therefore, $E(t,x)$ converges to 0.
Similarly,
$$
\partial_t (S(t,x)+E(t,x)+I(t,x)) +\gamma(x) (S(t,x)+E(t,x)+I(t,x)) = \gamma(x) (S(t,x)+E(t,x)).
$$
Appealing to the hypothesis (made in \eqref{eq:hyp}) that $\gamma(x) \geq \underline\gamma >0$, permits to show that $I(t,x)\to 0$ when $t\to +\infty$, a.e.~and uniformly on $\Omega$.}
This demonstrates \eqref{eq18} and achieves the proof of Lemma \ref{lem:lim}.
\end{proof}

\subsection{Herd immunity and basic reproduction number}
\label{se32}

Recall that \eqref{eq:hyp} and \eqref{eq:init} have been assumed. 
Following the idea in Ref.~\cite{DiekmannR0}, we define now a key notion.

\begin{definition}
For any {\color{black} $S(\cdot)\in L^1(\Omega) \cap L^\infty_+(\Omega)$,}
we call {\em next-generation operator} $K_{S(\cdot)}$ associated to \eqref{eqG1}, the operator defined over $L^2(\Omega)$ by:
\begin{equation}\label{def:Kt}
  K_{S(\cdot)}[\phi](x) = S(x) \int_\Omega \frac{\beta(x,y)}{\gamma(y)} \phi(y)\,dy.
\end{equation}
\end{definition}

Under assumptions \eqref{eq:hyp} and \eqref{eq:init}, we know that, for any $t\geq 0$, the solution $S(t,\cdot)$ of \eqref{eqG1} belongs to $L^\infty_+(\Omega)\cap  L^1(\Omega)$.
The operator $K_{S(t,\cdot)}$ is thus well-defined and it  is a positive and compact operator over $L^2(\Omega)$.
This permits to define its norm and its spectral radius in the usual way recalled now.
 
\begin{definition}
The {\em spectral radius} $r(K_{S(t, \cdot)})$ of $K_{S(t, \cdot)}$ is defined by
\[
  r(K_{S(t, \cdot)}) = \lim_{n\to +\infty} \|{K_{S(t, \cdot)}}^n\|^{1/n}.
\]
\end{definition}

As $\beta$ is irreducible thanks to \eqref{eq:hyp} and using Krein-Rutman theory\cite{Krein,Zeidler}, we can identify the spectral radius $r(K_{S(t, \cdot)})$ of $K_{S(t, \cdot)}$ with its principal eigenvalue.
The latter is simple, and defined as the unique eigenvalue associated to a positive eigenfunction.
{\color{black} In particular, under the assumptions of Theorem \ref{th1}, $r(K_{S(t, \cdot)})>0$ for any $t\geq 0$.}
We denote $\varphi_{S(t, \cdot)}$ this eigenfunction, so that by definition,
\[
  K_{S(t, \cdot)}[\varphi_{S(t, \cdot)}](x)=S(t,x) \int_\Omega \frac{\beta(x,y)}{\gamma(y)} \varphi_{S(t, \cdot)}(y)\,dy = r(K_{S(t, \cdot)}) \varphi_{S(t, \cdot)}(x).
\]
We  also know that $r(K_{S(t, \cdot)})=r(K_{S(t, \cdot)}^{\star})$, where $K_{S(t, \cdot)}^{\star}$ is the adjoint operator, defined by 
\[
\langle K_{S(t, \cdot)}^\star[\psi], \phi \rangle
= \langle \psi, K_{S(t, \cdot)}[\phi] \rangle
= \int_\Omega \psi(x) S(t,x) \int_\Omega \frac{\beta(x,y)}{\gamma(y)} \phi(y)\,dy\, dx
\]
that is
\begin{equation}  \label{def:Ktstar} 
  K_{S(t, \cdot)}^{\star}[\psi] (y):= \int_\Omega S(t,x)  \frac{\beta(x,y)}{\gamma(y)} \psi(x)\,dx.
\end{equation}
We let $\psi_{S(t, \cdot)}$ be the associated normalized adjoint eigenfunction,  defined by
\[
 K_{S(t, \cdot)}^\star [\psi_{S(t, \cdot)}] = r(K_{S(t, \cdot)}) \psi_{S(t, \cdot)},  \qquad \text{and} \qquad \int_\Omega \psi_{S(t, \cdot)}(x) S_{0}(x)dx=1.
\]

We are now in position to define for \eqref{eqG1} the key notions of  basic reproduction number and  herd immunity, based on the definition of next-generation operator in \eqref{def:Kt}.
\begin{definition}\label{def:R0}
Under assumptions \eqref{eq:hyp} and \eqref{eq:init}, we define:
  \begin{itemize}
  \item {\em basic reproduction number} $\calR_0$  the spectral radius of the operator $K_{S_0(\cdot)}$, i.e., $\calR_0 = r(K_{S_0(\cdot)})$;
  \item {\em herd immunity domain} the set of those elements $\bar S$ of $L^\infty_+(\Omega)\cap  L^1(\Omega)$ such that $r(K_{\bar S(\cdot)}) \leq 1$.
  \end{itemize}
\end{definition}
In contrast to the homogeneous SEIR system where herd immunity is reached for a uniquely defined proportion of the variable $S$, here the boundary of the herd immunity domain is a set of functions corresponding to various susceptible population distributions.
It is hard to determine in which point of the boundary the trajectory will enter this domain, but this necessarily takes place after a certain time, as stated now.

\begin{lemma}\label{lem:rKt}
  Assume that \eqref{eq:hyp} and \eqref{eq:init} hold, and that $r(K_{S_0(\cdot)})>1$.
  Then,  $t\mapsto r(K_{S(t, \cdot)})$ is decreasing and continuous with
\begin{equation}
\label{eq19}
  \lim_{t\to +\infty} r(K_{S(t, \cdot)}) < 1,
\end{equation}
and  there exists a unique $T_0>0$ such that $r(K_{S(T_0, \cdot)})=1$.
  As a consequence $S(t, \cdot)$ belongs to the herd immunity domain iff $t\geq T_0$.
\end{lemma}

\begin{proof}
We first establish that $t\mapsto r(K_{S(t, \cdot)})$ is non-increasing and continuous.
  This is a consequence of the fact that for a.e.\ $x\in\Omega$, $t\mapsto S(t,x)$ is decreasing and continuous. Indeed, by definition \eqref{def:Kt}, for any $\phi\in L^2(\Omega)$, and any $n\in \NN^*$, $t\mapsto \|{K_{S(t, \cdot)}}^n [\phi]\|$ is decreasing and continuous.
  Hence, $t\mapsto \|{K_{S(t, \cdot)}}^n\|^{1/n}$ is decreasing and continuous for any $n\in \NN^*$. The result follows by passing to the limit $n\to +\infty$.

  Secondly, since $r(K_{S_0(\cdot)})>1$, we are left to prove that \eqref{eq19} holds.
  From Lemma~\ref{lem:lim}, we may use the fact that $S_{\infty}(x)$ is the limit of $S(t,x)$ when $t\to +\infty$, for a.e.\ $x\in \Omega$.
Denote simply $K_{\infty} := K_{S_\infty(\cdot)}$ the associated next-generation operator, and by $\psi_{\infty}$ the principal eigenfunction associated to the adjoint operator $K_{\infty}^\star$, {\color{black} that is, $K_{\infty}^\star[\psi_\infty] = r(K_{\infty}^\star)\psi_\infty$. It follows from this identity that $r(K_{\infty}^\star)>0$.
Moreover one has $\psi_\infty>0$.
In addition, if \eqref{eq888} holds, then
$$\psi_\infty (x) \geq \frac{\underline{\beta}\ \ess\inf_\Omega S_\infty }{r(K_{\infty}^\star)\ess\sup_\Omega \gamma}\int_\Omega \psi_\infty (y)dy >0,$$
and in this second case, $\ess\inf_\Omega\psi_{\infty}>0$.}

%
  
 Adding the second and third equations in (\ref{eqG1}), multiplying  the result by $\psi_{\infty}$  and integrating, one obtains from \eqref{def:Ktstar} taken at the limit,
  \begin{align*}
      \ds \frac{d}{dt} \int_\Omega &\psi_{\infty}(x)\big( E(t,x)+I(t,x)\big)\,dx \\
      &= \int_{\Omega \times\Omega} \psi_{\infty}(x)S(t,x) \beta (x,y) I(t,y)\,dydx - \int_\Omega \gamma (x) \psi_{\infty}(x)I(t,x)\,dx \\
      &\geq   \int_{\Omega \times\Omega} \psi_{\infty}(x)S_{\infty}(x) \beta (x,y) I(t,y)dydx - \int_\Omega \gamma (x) \psi_{\infty}(x)I(t,x)\,dx \\[5pt]
      &= \langle K_\infty^\star[\psi_\infty], \gamma (\cdot) I(t,\cdot) \rangle - \langle \psi_\infty, \gamma (\cdot) I(t,\cdot) \rangle \\
      &= \big( r(K_{\infty}^\star)-1\big)\int_{\Omega} \psi_{\infty}(y)  \gamma (y)  I(t,y)\,dy.
  \end{align*}
  Hence, arguing by contradiction,  if $r(K^{\star}_{\infty})\geq 1$, then, on the one hand $t\mapsto \int_{x} \psi_{\infty}(x)\big( E(t,x)+I(t,x)\big)dx$ is non-decreasing. 
  On the other hand, we have seen in Lemma \ref{lem:lim} that $\lim_{t\to +\infty} \big( E(t,x)+I(t,x)\big) = 0$. It is a contradiction. Thus, $r(K^{\star}_{\infty})<1$ and by continuity there exists $T_0$ such that $r(K_{S(T_0, \cdot)})=1$.
This achieves the proof of Lemma~\ref{lem:rKt}.
\end{proof}

As is the case for the homogenous SEIR model, there exists a tight relation between herd immunity and stability of the equilibrium points of \eqref{eqG1}.
This link is elucidated in Section \ref{se34} below.

\subsection{Long time behaviour and final size equation -- Proof of Theorem \ref{th1}}
\label{se33}

{\color{black} We are now in position to achieve the proof of Theorem \ref{th1}.}

\paragraph{\it \color{black} Step 1.\ Preliminaries.}

In Lemma \ref{lem:lim} we already proved the decay of $t\mapsto S(t,x)$ and the existence of limits for all variables {\color{black} $S,E, I$ and $R$.
Basically we have still to achieve two tasks: identifying the limit $S_\infty$, and showing the exponential convergence in the case where \eqref{eq888} holds.}
Integrating \eqref{Lyap3} between $t$ and $+\infty$, we obtain that the latter satisfies necessarily
\begin{equation}
\label{eq38}
  \left\{
    \begin{array}{l}
      \ds \ln S_\infty(x) -  \Phi_\infty (x) = \ln S(t,x) -  \int_\Omega \frac{\beta (x,y)}{\gamma(y)} [S(t,y)+E(t,y)+I(t,y)]\,dy,
      \\[3mm]
      \ds \Phi_\infty (x) =  \int_\Omega \frac{\beta (x,y)}{\gamma(y)} S_\infty(y)\,dy,
    \end{array}\right.
\end{equation}
for any $t\geq 0$.
{\color{black} Taking $t=0$, yields an alternative form of \eqref{eq:Sinf}/\eqref{eq45} that will be useful in the sequel.}

\paragraph{\it Step 2. Exponential convergence properties.}

 {\color{black} We assume in this paragraph that \eqref{eq888} holds.
 From \eqref{eq45}, it is deduced directly that $\ess\inf_\Omega S_\infty>0$.

From the uniform convergence of $S(t,\cdot)$ towards $S_\infty(\cdot)$ previously ascertained in Lemma \ref{lem:lim}, one deduces that,}
for all  $\varepsilon >0$, there is a $T_\varepsilon $ such that $S(t,\cdot) \leq (1+ \varepsilon ) S_\infty $ for $t\geq T_\varepsilon $.
We also introduce $\theta \in (0,1)$, to be chosen later, and compute 
\begin{equation}
\label{eq88}
\frac{d}{dt} [E(t,x)+ \theta I(t,x)]\leq (1+ \varepsilon ) S_\infty (x) \int_\Omega  \beta(x,y) I(t,y)\ dy - (1-\theta) \underline \alpha E(t,x) - \theta  \gamma(x) I(t,x),
\end{equation}
where $\underline\alpha>0$ has been defined in \eqref{eq:hyp}.
Multiplying again by $\psi_\infty$, the principal eigenfunction associated to the adjoint operator $K_\infty^{\star}=K_{S_\infty(\cdot)}^{\star}$ introduced in the proof of Lemma \ref{lem:rKt}, and using~\eqref{def:Ktstar}, we find for $t\geq T_\varepsilon $
\begin{eqnarray*}
\lefteqn{\frac{d}{dt} \int_\Omega \psi_\infty (x) [E(t,x)+ \theta I(t,x) ]dx}\\
& \leq &
(1+ \varepsilon )  \int_{\Omega \times \Omega} \psi_\infty(x) S_\infty (x)  \beta(x,y) I(t,y)\ dydx 
-    \int_\Omega \psi_\infty (x) [(1-\theta)\underline \alpha E(t,x) + \theta \gamma (x) I(t,x) ] dx\\
& \leq &
(1+ \varepsilon ) r(K_\infty)  \int_\Omega  \psi_\infty(y) \gamma(y)  I(t,y)\ dy -   \int_\Omega \psi_\infty(x) [(1-\theta)\underline \alpha E(t,x) + \theta  \gamma (x) I(t,x) ]\,dx.
\end{eqnarray*}

Since we know from Lemma \ref{lem:rKt} that $r(K_\infty) <1$, {\color{black} one may} choose $\theta $ sufficiently close to $1$ and $\varepsilon $ small enough such that $1 > \theta > (1+ \varepsilon ) r(K_\infty)$, in such a way that
\[
\lambda:= \min \left( (1-\theta)\underline \alpha, \left (1 - \frac{ (1+ \varepsilon) r(K_\infty) }{\theta} \right)\underline \gamma \right) >0.
\]
We {\color{black} then} obtain 
\[
\frac{d}{dt}  \int_\Omega \psi_\infty (x) [E(t,x)+ \theta I(t,x) ] dx  \leq - \lambda  \int_\Omega \psi_\infty (x) [E(t,x)+ \theta I(t,x)] dx,
\]
and, applying Gronwall lemma, we find exponential decay to $0$ with rate  $\lambda$ for {\color{black} both positive integrals $\displaystyle\int_\Omega \psi_\infty (x) E(t,x) dx$ and $\int_\Omega \psi_\infty (x) I(t,x) dx$}.

{\color{black}
Now, as $\ess\inf_\Omega \psi_\infty>0$ (see the proof of Lemma \ref{lem:rKt}), one deduces that both $\int_\Omega E(t,x)dx$ and $\int_\Omega I(t,x)dx$ converge exponentially to $0$. It follows from (\ref{eqG1c}) that $E(t,x)$ converges exponentially to $0$ as $t\to +\infty$ for a.e. $x\in \Omega$. Similarly, one gets  from (\ref{eqG1b}) the exponential convergence of $I(t,x)$.

Last,} using the fact that $S(t,x)$ decreases along time, equation \eqref{eqG1a}
tells us that 
\begin{align*}
0 \leq S(t,x) - S_\infty(x) &= \int_t^\infty  S(s, x) \int_\Omega \beta(x,y) I(s,y)  \,dyds \\
 &\leq \frac{S^0(x) \| \beta \|_{L^\infty}}{\inf \psi_\infty} {\color{black} \int_t^{\infty}} \int_\Omega \psi_\infty (y) I(s,y)  \,dy ds  =O(e^{-\lambda t}).
\end{align*}

\paragraph{\it Step 3. Contraction property for $t$ large enough \color{black} and local uniqueness of the solution of \eqref{eq:Sinf}.}

From Lemma \ref{lem:rKt}, there exists $T>0$ such that $r(K_{S(T, \cdot)})<1$.
From the first line of equation \eqref{eq38}, we deduce the expression for $S_\infty$
\begin{equation}
\label{eq:Sinf2}
  S_\infty(x) = S(T,x) \exp (\Phi_\infty(x)-\Phi(T,x)),
\end{equation}
where we use the notation  \eqref{eq11}.
Injecting this expression into the second line of \eqref{eq38}, yields the following fixed point problem for the function $\Phi_\infty$~:
  \begin{equation}
    \label{eq:phiinf}
    \Phi_\infty(x) = \int_\Omega \frac{\beta (x,y)}{\gamma(y)} S(T,y) \exp(\Phi_\infty(y)-\Phi(T,y))\,dy.
  \end{equation}
  Notice that since $t\mapsto S(t,x)$ is decreasing from \eqref{eqG1a}, we deduce that $S_\infty\leq S(T,\cdot)$ and $\Phi_\infty \leq \Phi(T,\cdot)$.
  
Introduce
  \[
    \mathcal{F}(u)(x) := \int_\Omega \frac{\beta (x,y)}{\gamma(y)} S(T,y) \exp(u(y)-\Phi(T,y))\,dy,
  \]
in such a way that the fixed point problem \eqref{eq:phiinf} reads $\Phi_\infty = \mathcal{F}(\Phi_\infty)$.
  Let us denote 
  \[
    \mathcal{K}:=\{u\in L^1(\Omega), \text{ such that } 0\leq u \leq \Phi(T,\cdot) \hbox{ a.e.}\}.
  \]
  Clearly, $\mathcal{K}$ is stable by $\mathcal{F}$ (using the definition of $\Phi$ in \eqref{eq11}), and we may prove that $\mathcal{F}$ is a contraction on the set $\mathcal{K}$.
  Indeed, we compute, for any $u$, $v$ in $\mathcal{K}$,
  \begin{align}
    |\mathcal{F}(u) - \mathcal{F}(v)|(x)
    & \nonumber
\color{black}
\leq \int_\Omega \frac{\beta (x,y)}{\gamma(y)} S(T,y) \left|\exp(u(y)-\Phi(T,y))-\exp(v(y)-\Phi(T,y))\right|\,dy   \\
    & \label{eq53}
    \leq \int_\Omega \frac{\beta (x,y)}{\gamma(y)} S(T,y) |u(y)-v(y)|\,dy.
  \end{align}
  {\color{black}
To deduce the last inequality, we used the fact that $u(y)-\Phi(T,y)\leq 0$, $v(y)-\Phi(T,y)\leq 0$ for almost every $y\in \Omega$ (as a consequence of the fact that $u,v\in\mathcal{K}$ and by definition of this set) and applied the mean value theorem.}
  
  At this stage, we introduce on the set $\mathcal{K}$ the norm 
  $$
   \|u\|_T:= \int_\Omega \psi_{S(T, \cdot)}(x) S(T,x) |u(x)|\,dx,
  $$
where, as defined in~\eqref{def:Ktstar}, $\psi_{S(T, \cdot)}$ is the principal eigenfunction of the adjoint $K_{S(T, \cdot)}^\star$.
(Notice that, when $\ess\inf_\Omega\psi_{S(T,\cdot)}>0$ by the same arguments as in Lemma \ref{lem:rKt}, the norm $\|\cdot\|_T$ is equivalent to $L^\infty$ norm, and $({\cal K}, \|\cdot\|_T)$ is a Banach space; but this property will not be necessary in the sequel.)
Multiplying by $\psi_{S(T, \cdot)}$ and $S(T,\cdot)$, we get with the help of \eqref{eq53}, that
  \begin{eqnarray*}
\|\mathcal{F}(u) - \mathcal{F}(v)\|_T
      & = &
    \int_\Omega \psi_{S(T, \cdot)}(x) S(T,x) |\mathcal{F}(u)(x) - \mathcal{F}(v)(x)|\,dx\\
    & \leq &
     \int_{\Omega\times\Omega} \psi_{S(T, \cdot)} (x) \frac{\beta (x,y)}{\gamma(y)} S(T,x) S(T,y) |u(y)-v(y)|\,dydx \\ 
     & = &
     r(K_{S(T, \cdot)}) \int_\Omega \psi_{S(T, \cdot)} (y) S(T,y) |u-v|(y)dy = r(K_{S(t, \cdot)}) \|u-v\|_T,
  \end{eqnarray*}
  by definition 
  of the norm~$\|\cdot\|_T$.
  As $r(K_{S(T, \cdot)})<1$, we deduce that $\mathcal{F}$ is a contraction on $\mathcal{K}$.
  {\color{black}
The existence of the limit $S_\infty$ has been established previously (see Lemma \ref{lem:lim}), together with the fact that it belongs to the set $\cal K$.
The contraction inequality now shows that $\Phi_\infty$ is the unique solution to \eqref{eq:phiinf} in the set $\mathcal{K}$, from which $S_\infty$ is deduced, via \eqref{eq:Sinf2}.
}

\paragraph{\it Step 4. Uniqueness below $S_0$.}
 It remains to extend the uniqueness result to all function smaller or equal to $S_0$.
 We make use of the following result.

 \begin{lemma} 
 \label{le7}
 \color{black} The solution $S_\infty$ is larger than all sub-solutions of \eqref{eq:Sinf}.
 \end{lemma}

 Indeed, we deduce from Lemma \ref{le7} that any other solution $\underline S$ of \eqref{eq:Sinf} satisfies $\underline S \leq S_\infty$.
 Therefore it also satisfies $\underline S \leq S(T)$ for $T$ large enough.
 Thus by the contraction property of Step 3 it is equal to $S_\infty$.
This achieves the demonstration of Theorem \ref{th1}, with the exception of the proof of Lemma \ref{le7}, which we now complete.

 \begin{proof}[Proof of Lemma \ref{le7}]
 We first show that $S_\infty$ is the limit of the sequence $S_k(x)$ departing from the initial condition $S_0$ and constructed, for $k \geq 0$, by the following iteration:
 \begin{align*}
   &\ln S_{k+1}(x) = \int_{\Omega} \frac{ \beta (x,y) }{\gamma(y)} S_{k}(y)\,dy +A(x), \\
   & A(x)= \ln S_0(x) -  \int_{\Omega} \frac{ \beta (x,y) }{\gamma(y)} [S_{0} +E_{0}+I_{0}](y)\,dy.
 \end{align*}
 Obviously, we have
   \[
     \ln S_{1}(x) = \int_{\Omega} \frac{ \beta (x,y) }{\gamma(y)} S_{0}(y)\,dy + A(x)
     = \ln S_{0}(x) -  \int_{\Omega} \frac{ \beta (x,y) }{\gamma(y)} [E_{0}+I_{0}](y)\,dy
     < \ln S_{0}(x).
   \]
   And iterating, if we know that $S_k < S_{k-1}$, we find 
   \[
     \ln S_{k+1}(x) = \int_{\Omega} \frac{ \beta (x,y) }{\gamma(y)} S_{k}(y)\,dy + A(x) < \int_{\Omega} \frac{ \beta (x,y) }{\gamma(y)} S_{k-1}(y)\,dy + A(x) =\ln S_{k}(x) ,
   \]
   which gives $S_{k+1}<S_{k}$. As a consequence $S_k$ is a decreasing sequence which converges to a solution  $\bar S$ of \eqref{eq:Sinf}.

   Secondly, let $\underline S\leq S_0$ be a sub-solution of \eqref{eq:Sinf}. We have 
   \[
     \ln \underline S(x) \leq  \int_{\Omega} \frac{ \beta (x,y) }{\gamma(y)} \underline S(y)\,dy + A(x) \leq  \int_{\Omega} \frac{ \beta (x,y) }{\gamma(y)} S_{0}(y)\,dy + A(x) = \ln S_1(x),
   \]
   therefore $\underline S\leq S_1$, and iterating the argument $\underline S\leq S_k$ for all $k \in \mathbb N$ and thus $\underline S\leq \bar S$. In particular we have $S_\infty \leq \bar S$.

   Thirdly, we prove that $ \bar S  \leq  S_\infty$. We know that
   $\bar S < S_{0}$ and we are going to prove that $\bar S < S(t)$ for all $t \geq0$ from which we conclude, by definition of $S_\infty$, that $\bar S \leq S_\infty$ and thus that $S_\infty = \bar S$. To prove this, we consider, by contradiction, the first time $t_0$ when   $S(t_0)$ touches  $\bar S $, i.e., $\bar S  \leq  S(t_0)$ and for some $x_0 \in \Omega$, $\bar S(x_0)  =  S(t_0,x_0)$.
Using the fact, proved in Lemma \ref{lem:conserv2}, that the quantity $\ln S(t,x)-\Phi(t, x)$ is constant with respect to time, we conclude that
   \[
     0= \ln  \bar S(x_0) - \ln S(t_0,x_0) = \int_{\Omega} \frac{ \beta (x,y) }{\gamma(y)} [ \bar S(y) - S(t_0, y)- E(t_0, y)-I(t_0,y)]\,dy <0,
   \]
   a contradiction.
   This achieves the demonstration of Lemma \ref{le7}.
 \end{proof}

\subsection{Herd immunity and stability of the equilibrium points}
\label{se34}

Before closing Section \ref{sec:gen} and the study of system \eqref{eqG1}, we analyze in the following result the link between herd immunity and the stability of the equilibria of this system.

\begin{lemma}
\label{le777}
Assume \eqref{eq:hyp} and \eqref{eq:init} hold.
  The equilibrium points of system \eqref{eqG1} are exactly the configurations $(S^*,0,0,R^*)$ that fulfil \eqref{eq:init}.
Moreover each of them is stable if
  $
  r(K_{S^*(\cdot) }) < 1,
  $
  unstable if
  $
    r(K_{S^*(\cdot) }) > 1,
  $
\end{lemma}
\begin{proof}
The equilibria for \eqref{eqG1} are any configuration $(S^*,E^*,I^*,R^*)$ such that, for a.e.\ $x\in\Omega$,
\begin{gather*}
-S^*(x) \int_\Omega \beta(x,y) I^*(y)\ dy = 0,\qquad S^*(x) \int_\Omega  \beta(x,y) I^*(y)\ dy - \alpha(x) E^*(x) = 0,\\
\alpha (x) E^*(x) - \gamma(x) I^*(x) = 0,\qquad \gamma(x) I^*(x)=0.
\end{gather*}
We deduce from the last two equations that $(E^*,I^*) = (0,0)$.
This condition is also sufficient to have an equilibrium, provided that it fulfils \eqref{eq:init}.

The stability of such an equilibrium is related to the stability of the linearized system
\begin{subequations}
\label{eq98}
\begin{gather}
\dfrac{\partial s}{\partial t} = -S^*(x) \int_\Omega \beta(x,y) i(t,y)\ dy,\quad
\dfrac{\partial e}{\partial t}  = S^*(x) \int_\Omega \beta(x,y) i(t,y)\ dy - \alpha(x) e(t,x),\\
  \dfrac{\partial i}{\partial t}  = \alpha(x) e(t,x) - \gamma(x) i(t,x),\qquad
    \dfrac{\partial r}{\partial t}  =  \gamma(x) i(t,x).
\end{gather}
\end{subequations}
Notice that \eqref{eq98} admits a block-triangular structure: the evolution of $(e,i)$ is independent of the other two components $s$ and $r$; while the evolution of $s$ and $r$ only depends upon $i$.

Assume first that $r(K_{S^*(\cdot)})<1$.
Let $\psi_{S^*(\cdot)}$ be the principal adjoint eigenvector of the next-generation operator $K_{S^*(\cdot)}$ associated to $S^*$.
Mimicking the computation achieved in Step 2 of Section \ref{se33}, one finds that, for any $\theta\in (0,1)$ and a.e.\ $x\in\Omega$,
\begin{eqnarray*}
\lefteqn{\frac{d}{dt} \int_\Omega \psi_{S^*(\cdot)}(x) [e(t,x)+\theta i(t,x)]\ dx}\\
& = &
\int_{\Omega\times\Omega} \psi_{S^*(\cdot)}(x) S^*(x)\beta(x,y) i(t,y)\ dy  \\
&&\qquad\qquad -  \int_\Omega \psi_{S^*(\cdot)}(x) [(1-\theta) \alpha(x) e (t,x) + \theta  \gamma (x) i(t,x) ]\,dx \\
& \leq &
r(K_{S^*(\cdot)})  \int_\Omega  \psi_{S^*(\cdot)}(y) \gamma(y) i(t,y)\ dy\\
&&\qquad\qquad 
- \int_\Omega \psi_{S^*(\cdot)}(x) [(1-\theta)\underline \alpha e(t,x) + \theta  \gamma (x) i(t,x) ]\,dx.
\end{eqnarray*}
Taking $\theta \in (r(K_{S^*(\cdot)}),1)$ then yields (through application of Gronwall's lemma) exponential decrease of $\int_\Omega \psi_{S^*(\cdot)}(x) [e(t,x)+\theta i(t,x)]\ dx$ when $t\to +\infty$.
One sees easily that the subsystem in $(e,i)$ of \eqref{eq98} is a monotone system.
The eigenfunction $\psi_{S^*(\cdot)}$ being positive, the functional $(e(\cdot),i(\cdot))\mapsto \int_\Omega \psi_{S^*(\cdot)}(x) [e(\cdot)+\theta i(t\cdot)]\ dx$ is thus a Lyapunov functional for this subsystem, and the origin of the latter is asymptotically stable.
Due to the block-triangular structure mentioned earlier,  the complete system \eqref{eq98} is (simply) stable.

Assume now that $r(K_{S^*(\cdot)})>1$.
Denoting $\varphi_{S^*(\cdot)}$ the principal eigenvector of $K_{S^*(\cdot)}$, define
$$
e^*(x) := \frac{r(K_{S^*(\cdot)})+1}{2}\frac{\varphi_{S^*(\cdot)}(x)}{\alpha(x)},\qquad
i^*(x) := \frac{\varphi_{S^*(\cdot)}(x)}{\gamma(x)}.
$$
Let us evaluate the value of the right-hand sides corresponding to $\frac{\partial e}{\partial t}$ and $\frac{\partial i}{\partial t}$ at the point $(e^*,i^*)$.
Then, for any $x\in\Omega$,
\begin{eqnarray*}
\lefteqn{S^*(x) \int_\Omega \beta(x,y) i^*(y)\ dy - \alpha(x) e^*(x)}\\
& = &
S^*(x) \int_\Omega \beta(x,y) \frac{\varphi_{S^*(\cdot)}(y)}{\gamma(y)}\ dy - \frac{r(K_{S^*(\cdot)})+1}{2} \varphi_{S^*(\cdot)}(x)  \\
& = & K_{S^*(\cdot)}[\varphi_{S^*(\cdot)}](x) - \frac{r(K_{S^*(\cdot)})+1}{2} \varphi_{S^*(\cdot)}(x)\\
& = &
\frac{r(K_{S^*(\cdot)})-1}{2} \varphi_{S^*(\cdot)}(x)
> \frac{r(K_{S^*(\cdot)})-1}{r(K_{S^*(\cdot)})+1} \underline\alpha e^*(x),
\end{eqnarray*}
while similarly 
\begin{eqnarray*}
\alpha(x) e^*(x) - \gamma(x) i^*(x)
& = &
\frac{r(K_{S^*(\cdot)})+1}{2} \varphi_{S^*(\cdot)}(x) - \varphi_{S^*(\cdot)}(x)\\
& = &
\frac{r(K_{S^*(\cdot)})-1}{2} \varphi_{S^*(\cdot)}(x)
> \frac{r(K_{S^*(\cdot)})-1}{2}\underline\gamma i^*(x).
\end{eqnarray*}
The subsystem describing the evolution of the components $(e,i)$ being a monotone system, this condition is sufficient to ensure that any trajectory of this subsystem departing from a point $(e_0,i_0) \geq \lambda (e^*,i^*)$, $\lambda>0$, is increasing.
This shows the instability of the origin of \eqref{eq98}, and achieves the proof of Lemma \ref{le777}.
\end{proof}

\section{General contact matrices with diffusion}
\label{sec:genDiff}

We now come to system~\eqref{eqGLaplacian} where diffusion is included, and prove Theorem~\ref{th2}. We adapt for this the methodology developed in Section~\ref{sec:gen}.

\subsection{Conserved quantities and long time limit}
\label{se41}

Let us introduce the solution of the nonlinear elliptic equation
\begin{equation}\label{eq:Phi}
  \left\{
\begin{array}{l}
-\Delta \Phi(t,x) +\gamma \Phi(t,x) = S(t,x)+E(t,x)+I(t,x), \qquad \text{on } \Omega,
\\[2pt]
 \partial_{n}\Phi(t,x) = 0 \quad \hbox{ on }  (0,\infty)\times \partial\Omega,
 \end{array} \right.
\end{equation}
where $n$ is the outward normal unit at the boundary of $\Omega$.

\begin{lemma}\label{lem:conserv3}
  Under assumption \eqref{eq:hyp}, the solution of \eqref{eqGLaplacian} verifies that the quantity $\int_\Omega (S(t,x) + E(t,x) + I(t,x) + R(t,x))\,dx$ is constant with respect to time, as well as the quantities $\ln S(t,x) - \int_{\Omega} \beta(x,y) \Phi(t,y)\,dy$ for a.e.\ $x\in \Omega$.
  \end{lemma}
\begin{proof}
  The first conservation relation is obtained straightforwardly by adding and integrating on $\Omega$ the equations in \eqref{eqGLaplacian}.
  For the second, we compute, as usual,
  \begin{equation}
    \partial_t\ln S(t,x) = -\int_{\Omega} \beta(x,y) I(t,y)\ dy.
    \label{lnSL}
  \end{equation}
  Secondly
  \[
    -\Delta  \partial_t \Phi(t,x) +\gamma(x)\partial_t \Phi = \frac{ \partial}{\partial t} (S+E+I)(t,x) =  - \gamma(x) I(t,x) + \Delta I(t,x)  .
  \]
Therefore, by the uniqueness of solution of \eqref{eq:Phi}, which is ensured by the fact that $\underline\gamma >0$ (see \eqref{eq:hyp}), we obtain
  \[
    \partial_t \Phi(t,x)  = -I(t,x).
  \]
Inserting this expression of $I$ in equation \eqref{lnSL} then yields
  \begin{equation}
    \label{eq:Lyap}
    \frac{\partial}{\partial t} \left[-\ln S(t,x) +\int_{\Omega} \beta(x,y) \Phi(t,y) dy \right] =0.
  \end{equation}
  This achieves the proof of Lemma \ref{lem:conserv3}.
\end{proof}

As before, we may deduce existence of limits from the previous result, as stated now.
\begin{lemma}\label{lem:lim2}
  Assume that \eqref{eq:hyp} and \eqref{eq:init} hold. Then the solution of \eqref{eqGLaplacian} verifies~:

  There exist $S_\infty,R_\infty\in L^\infty_+(\Omega)$, $S_\infty\leq S_0$, such that, 
  $$
  \lim_{t\to + \infty} S(t,x) = S_\infty(x),
\lim_{t\to + \infty} R(t,x) = R_\infty(x),
  \text{ and }
\lim_{t\to + \infty} E(t,x)= \lim_{t\to + \infty} I(t,x) = 0,
  $$
 uniformly, for a.e.\ $x\in\Omega$.

  Moreover, for a.e.\ $x\in \Omega$, $\Phi(t,x)$ converges to $\Phi_\infty(x)$ as $t$ goes to $+\infty$ where
\begin{equation}\label{eq:Phiinf}
  \left\{
    \begin{array}{l}
      -\Delta \Phi_\infty(x) +\gamma \Phi_\infty(x) = S_\infty(x),
      \\[2pt]
      \partial_{n}\Phi_\infty(x) = 0 \quad \hbox{ on }  (0,\infty)\times \partial\Omega.
    \end{array}
  \right.
\end{equation}
\end{lemma}
\begin{proof} As above, for a.e.\ $x\in\Omega$, since $t\mapsto S(t,x)$ is decreasing, it has a limit $S_\infty (x)$, and because $t\mapsto S(t,x)+E(t,x)$ is decreasing, $t\mapsto E(t,x)$ has a limit which vanishes as can be seen adding from \eqref{eqGLaplacianS}--\eqref{eqGLaplacianE}.
\textcolor{magenta}{By the dominated convergence theorem, one deduces also that these convergences hold in $L^1(\Omega)$. 
Then it follows from equations \eqref{eqGLaplacianS} and \eqref{eqGLaplacianE} by application of Arzela-Ascoli's theorem that the convergences of $S(t,\cdot)$ towards $S_\infty$ and of $E(t,\cdot)$ towards $0$ is uniform a.e.~over $\Omega$.}

Inserting this information in equation~\eqref{eqGLaplacianI}, we conclude that, \textcolor{magenta}{$I(t,\cdot)$ converges towards $0$ uniformly a.e.~over $\Omega$ as $t\to +\infty$.}
  Finally, we pass to the limit in equation \eqref{eq:Phi} to get the limit $\Phi_\infty$ of $\Phi$.
\end{proof}

\subsection{The next-generation operator}
\label{se42}

Following the approach of Section \ref{sec:gen}, we define the next-generation operator.

\begin{definition}
For any $t\geq 0$ and {\color{black} $S(t,\cdot)\in L^1(\Omega) \cap L^\infty_+(\Omega)$} solution of \eqref{eqGLaplacian}, we call {\em next-generation operator} $K_{S(t, \cdot)}$ associated to \eqref{eqGLaplacian}, the operator defined over $L^2(\Omega)$ by
\begin{equation}\label{eq:Kdelta}
 K^{\Delta}_{S(t,\cdot)}[ \phi ](x):=S(t,x) \int_{\Omega} \frac{\beta(x,y)}{\gamma(y)} \phi(y)\,dy + \Delta \left(\frac{ \phi(x)}{\gamma (x)}\right),
\end{equation}
with Neumann boundary conditions.
Its adjoint is defined as
\begin{equation}
\label{eq89}
  K^{\Delta*}_{S(t,\cdot)}[\psi](y):= \int_\Omega S(t,x) \frac{\beta(x,y)}{\gamma(y)} \psi(x)\,dx + \frac{1}{\gamma (y)} \Delta \psi(y),
\end{equation}
also with Neumann boundary conditions. 
\end{definition}

 
 The existence of a principal eigenvalue for the next-generation operator is established in the next result, whose proof is postponed to \ref{sec:appKdelta}.
 
 \begin{lemma}\label{lem:Kdelta}
   Let {\color{black} $S\in L^1(\Omega) \cap L^\infty_+(\Omega)$.}
   The operator $K^{\Delta}_{S(\cdot)}$ admits a principal eigenvalue, that is, there exists $r(K_{S(\cdot)}^{\Delta}) > 0$ and functions $\varphi_{S(\cdot)} \in L^2_+(\Omega)$ and $\psi_{S(\cdot)}\in L^2_+(\Omega)$ such that
   \[
     K^\Delta_{S(\cdot)} [\varphi_{S(\cdot)}]  = r(K_{S(\cdot)}^{\Delta})  \varphi_{S(\cdot)}, \quad \text{ in } \Omega, \qquad \partial_n \varphi_{S(\cdot)} = 0, \quad \text{ on } \partial\Omega ,
   \]
   and
   \[
     K^{\Delta*}_{S(\cdot)} [\psi_{S(\cdot)}] = r(K_{S(\cdot)}^{\Delta})  \psi_{S(\cdot)}, \quad \text{ in } \Omega, \qquad \partial_n \psi_{S(\cdot)} = 0, \quad \text{ on } \partial\Omega.
   \]
\end{lemma}

The basic reproduction number and the herd immunity domain may then be defined as in Definition \ref{def:R0}.
Notice that an analogue of Lemma \ref{le777}, which links herd immunity and stability, may also be written when diffusion is present.

From Lemma \ref{lem:lim2}, there exists {\color{black} $S_\infty\in L^1(\Omega) \cap L^\infty_+(\Omega)$,} which is the limit of $S(t,\cdot)$. We use the shorthand notation
$K^{\Delta}_\infty := K^{\Delta}_{S_\infty(\cdot)}$ and $K^{\Delta*}_\infty := K^{\Delta*}_{S_\infty(\cdot)}$; its principal eigenvalue is denoted $r^\Delta_\infty:=r(K^{\Delta}_{S_\infty(\cdot)})$ and the associated eigenfunctions are respectively $\varphi_\infty := \varphi_{S_\infty(\cdot)}$ and $\psi_\infty := \psi_{S_\infty(\cdot)}$.

The properties described in Lemma \ref{lem:rKt} still hold in this framework:
\begin{lemma}\label{lem:rKdelta}
Assume that \eqref{eq:hyp}, \eqref{eq:init} hold, and that $r(K^\Delta_{S_0(\cdot)}) > 1$. Then, $t\mapsto r(K^\Delta_{S(t,\cdot)})$ is decreasing, continuous; $\lim_{t\to +\infty} r(K^\Delta_{S(t,\cdot)}) < 1$ and there exists a unique $T_0>0$ such that $r(K^\Delta_{S(T_0,\cdot)}) =1$.
\end{lemma}
\begin{proof}
  The first point follows exactly the proof of Lemma \ref{lem:rKt}.
  For the existence of $T_0$, as above, we multiply \eqref{eqGLaplacianE} and \eqref{eqGLaplacianI} by $\psi_{\infty}$, and integrating, one finds 
\begin{align*}
\frac{d}{dt} &\int_{\Omega} \psi_{\infty}(x)\big( E(t,x)+I(t,x)\big)\,dx  \\
             &= \int_{\Omega\times \Omega} \psi_{\infty}(x) S(t,x) \beta (x,y) I(t,y)\,dydx - \int_{\Omega}\gamma (x) \psi_{\infty}(x)I(t,x)\,dx \\
  & \qquad + \int_{\Omega}\psi_{\infty}(x)\Delta I(t,x)\,dx  \\
&= \int_{\Omega\times \Omega} K^{\Delta*}_{S(t,\cdot)}[ \psi_{\infty}](y) \, \gamma(y) I(t,y)\,dy - \int_{\Omega} \psi_{\infty}(x) \gamma (x) I(t,x)\,dx  \\
&= \big( r_{\infty}^\Delta -1\big)\int_{\Omega} \gamma (x) \psi_{\infty}(x)I(t,x)\,dx.
\end{align*}
As in the proof of Lemma \ref{lem:rKt}, this proves that there is $T_0>0$ such that herd immunity is reached, that is $r(K_{S(T_0,\cdot)}^{\Delta})=1$. 
\end{proof}

\subsection{Long time behavior and final size equation -- Proof of Theorem \ref{th2}}
\label{se43}

We have now the material to prove Theorem \ref{th2}.
We follow and adapt to the case at hand the proof in Section \ref{sec:gen}.

\paragraph{\it Step 1. Existence of limits.}
The long time convergence has been studied in Lemma \ref{lem:rKt}.
Then, integrating \eqref{eq:Lyap} between $t$ and $+\infty$, we get, for a.e. $x \in \Omega, \; t \geq 0$.
\begin{equation}
  \ln S(t,x) - \int_{\Omega} \beta(x,y) \Phi(t,y)\,dy = \ln S_\infty(x) - \int_{\Omega} \beta(x,y) \Phi_\infty(y)\,dy.
  \label{energyL}
\end{equation}
It may be rewritten straightforwardly
\begin{equation}
\label{eq96}
  S_\infty(x) = S(t,x) \exp \left(\int_{\Omega} \beta(x,y)(\Phi_\infty(y)-\Phi(t,y))\,dy\right),
  \quad \text{ for a.e.\ } x \in \Omega, \; t \geq 0.
\end{equation}
Injecting this identity into the equation for $\Phi_\infty$, and taking $t=0$, we deduce that the long time limit is characterized by \eqref{eq:Phi_inf}.

\paragraph{\it Step 2. Exponential rate of decay.}
When \eqref{eq888} holds, exponential convergence also follows, as in Section~\ref{sec:gen}.
We just indicate the main steps.  Firstly, from~\eqref{energyL}, and because $\Phi(t,x) $ converges to $\Phi_\infty(x)$ almost everywhere, we conclude that for all $\varepsilon$ and for $t \geq T_\varepsilon$, we have $S(t,x) \leq (1+\varepsilon) S_\infty(x)$.

Secondly, adapting formula \eqref{eq88} to the present setting and using the definition of the adjoint operator $K^{\Delta*}_{S(t,\cdot)}$ given in \eqref{eq89}, we compute, for $\theta <1$ but close to $1$
\begin{align*}
&\frac{d}{dt} \int_{\Omega} \psi_{\infty}(x)\big( E(t,x)+ \theta I(t,x)\big)dx \\
&\leq (1+\varepsilon)  \int_{\Omega\times \Omega} \psi_{\infty}(x) S_\infty(x) \beta (x,y) I(t,y)dydx
-  (1-\theta)  \int_\Omega \psi_\infty (x)  \alpha(x) E(t,x) dx  
 \\
& \qquad  - \theta \int_{\Omega} \psi_{\infty}(x) \gamma (x) I(t,x)dx
+ \theta \int_{\Omega}\psi_{\infty}(x)\Delta I(t,x)dx 
\\
&\leq \theta  \int_{\Omega\times \Omega} K^{\Delta*}_{S(t,\cdot)}[ \psi_{\infty}](y) \, \gamma(y) I(t,y)dy
- (1-\theta)  \int_\Omega \psi_\infty (x)  \alpha(x) E(t,x) dx
 \\
&  \qquad - \theta \int_{\Omega} \psi_{\infty}(x) \gamma (x) I(t,x)dx
 + (1+\varepsilon -\theta) \int_{\Omega\times \Omega} \psi_{\infty}(x) S_\infty(x) \beta (x,y) I(t,y)dydx
\\
& \leq - \theta (1-r(K_{S(t,\cdot)}^{\Delta})) \int_{\Omega} \psi_{\infty}(x) \gamma (x) I(t,x)dx 
- (1-\theta)  \int_\Omega \psi_\infty (x)  \alpha(x) E(t,x) dx
 \\
&  \qquad 
+ (1+\varepsilon -\theta) \int_{\Omega\times \Omega} \psi_{\infty}(x) S_\infty(x) \beta (x,y) I(t,y)dydx.
\end{align*}
Taking $t >T_0$ in such a way that $r(K_{S(t,\cdot)}^{\Delta})<1$, choosing $\varepsilon>0$ small enough and $0<\theta<1$ close enough to $1$ to have $1+\varepsilon -\theta<0$, and using the fact (assumed in \eqref{eq:hyp}) that $\inf_{x\in \Omega} \alpha(x)$ and $\inf_{x\in \Omega} \gamma(x)$ are positive together with the positivity of the principal eigenvector $\psi_\infty$, allows to use again Gronwall lemma, and to conclude to the exponential rate of convergence of \textcolor{magenta}{$\int_\Omega \psi_\infty(x) E(t,x)\,dx$ and $\int_\Omega \psi_\infty(x) I(t,x)\,dx$.
  In the same spirit as in the proof of Lemma \ref{lem:rKt}, we have $r(K_\infty^{\Delta*})>0$ and from the relation $K^{\Delta*}_\infty[\psi_\infty] = r(K^{\Delta*}_\infty) \psi_\infty$, we deduce
  $$
  r(K_\infty^{\Delta*})\psi_\infty - \frac{1}{\gamma} \Delta \psi_\infty \geq \frac{\underline{\beta}\ \ess\inf_\Omega S_\infty}{\ess\sup_\Omega \gamma} \int_\Omega \psi_\infty(y)\,dy >0.
  $$
  By comparison principle, we get $\ess\inf_\Omega \psi_\infty>0$.
  Hence, we have the exponential convergence of $\int_\Omega E(t,x)\,dx$ and $\int_\Omega I(t,x)\,dx$. From \eqref{eqGLaplacianE}, we deduce the exponential convergence of $E(t,x)$ towards $0$ as $t\to +\infty$ for a.e. $x\in\Omega$. And from \eqref{eqGLaplacianI}, we have the exponential convergence of $I(t,x)$.
  Then, we may define, for a.e. $x\in\Omega$,
  $$
  R_\infty(x) = R_0(x) + \int_0^{+\infty} I(s,x)\,ds.
  $$
  From \eqref{eqGLaplacianR}, we have
  $$
  R_\infty(x) - R(t,x) = \int_t^{+\infty} I(s,x)\,ds.
  $$
  Then, the exponential convergence of $I$ implies the exponential convergence of $R(t,x)$ to $R_\infty(x)$ for a.e. $x\in \Omega$ as $t\to +\infty$.}
Exponential convergence of $S$ is then obtained as in Section~\ref{sec:gen}.
\textcolor{magenta}{It concludes the proof of the convergences in \eqref{eq18} at exponential rate.}

\paragraph{\it Step 3. Uniqueness for $t$ large enough.}
From Lemma \ref{lem:rKdelta} we may choose $T>0$ large enough such that $r(K^\Delta_{S(T,\cdot)})<1$.
From \eqref{eq96} and \eqref{eq:Phiinf}, we deduce
\begin{equation*}
  -\Delta \Phi_\infty(x) + \gamma(x) \Phi_\infty(x) = S(T,x) \exp\left(\int_\Omega \beta(x,y) (\Phi_\infty(y)-\Phi(T,y))\,dy\right),
\end{equation*}
complemented with Neumann boundary conditions.
Recall that $\Phi(t,\cdot)$ is defined for any $t$ by \eqref{eq:Phi}.
We notice that since $t\mapsto S(t,x)$ is decreasing, we have $S_\infty\leq S(T)$ and by the maximum principle applied to \eqref{eq:Phi}, we deduce also that $\Phi_\infty \leq \Phi(T)$.

Let us consider the problem
\begin{equation}
\label{eq97}
\left\{ \begin{array}{l}
  -\Delta \Big(\frac{u(x)}{\gamma(x)}\Big) + u(x) = S(T,x) \exp\left(\int_\Omega \beta(x,y) \Big(\frac{u(y)}{\gamma(y)}-\Phi(T,y)\Big)\,dy\right),
   \\[5pt]
     \partial_{n} \big(\frac{u(x)}{\gamma(x)}\big)=0 \quad \hbox{ on } \quad (0,\infty)\times \partial\Omega.
\end{array} \right.
\end{equation}

We show the uniqueness of a solution of this problem in the set $\{ 0\leq u \leq \gamma \Phi(T,\cdot) \text{ a.e.}\}$.
Let us assume that there are two solutions $u$ and $v$. 
We have as in Section \ref{se33}, using the fact that the exponential is $1$-Lipschitz in the set of nonpositive real numbers,
\begin{align*}
- \Delta  \Big(\frac{(u-v)(x)}{\gamma(x)}\Big) +  (u - v)(x) \leq S(T,x) \int_\Omega \frac{\beta(x,y)}{\gamma(y)}|u(y)-v(y)|\,dy.
\end{align*}
From Kato's inequality\cite{Kato}, we deduce
\begin{align}\label{ineq1}
- \Delta  \Big(\frac{|u-v|(x)}{\gamma(x)}\Big) +  |u - v|(x) \leq S(T,x) \int_\Omega \frac{\beta(x,y)}{\gamma(y)}|u(y)-v(y)|\,dy.
\end{align}

We recall that, from Lemma \ref{lem:Kdelta}, there exists $\psi_{S(T,\cdot)}$ such that 
\begin{align*}
  K^{\Delta*}_{S(T,\cdot)}[\psi_{S(T,\cdot)}](y) & = \int_\Omega S(T,x) \frac{\beta(x,y)}{\gamma(y)}\psi_{S(T,\cdot)}(x)\,dx + \frac{1}{\gamma(y)} \Delta \psi_{S(T,\cdot)}(y)  \\
  & = r(K^\Delta_{S(T,\cdot)}) \psi_{S(T,\cdot)}(y).
\end{align*}
We use the norm
\[
  \|u\|_T := \int_\Omega \psi_{S(T,\cdot)}(x) |u(x)|\,dx.
\]
Multiplying \eqref{ineq1} by $\psi_{S(T,\cdot)}$ and integrating, we get
\begin{align*}
  \int_\Omega \psi_{S(T,\cdot)}(x) |u-v|(x)\,dx
  \leq & \int_{\Omega\times\Omega} S(T,x) \psi_{S(T,\cdot)}(x) \frac{\beta(x,y)}{\gamma(y)}|u(y)-v(y)|\,dydx \\
       & + \int_\Omega \Delta \psi_{S(T,\cdot)}(x)  \Big(\frac{|u-v|(x)}{\gamma(x)}\Big)\,dx  \\
  \leq & \ r(K^\Delta_{S(T,\cdot)}) \|u-v\|_T.
\end{align*}
Since $r(K^\Delta_{S(T,\cdot)})<1$, we conclude to the uniqueness of the solution of \eqref{eq97}  in the set $\{0\leq u \leq \gamma \Phi(T,\cdot)\}$, and thus to the uniqueness of $\Phi_\infty$, obtained by dividing by $\gamma$, in the set $\{0\leq \Phi \leq  \Phi(T,\cdot)\}$.

\paragraph{\it Step 4. Uniqueness below $\Phi_0$.}  
It remains to extend the uniqueness result to all functions smaller than $\Phi_0$.
Analogously to the argument conducted in Section \ref{sec:gen}, we make use of the following result.
\begin{lemma}
Among the functions $\Phi\leq \Phi_0$, there is a solution $\bar\Phi$ of \eqref{eq:Phi_inf} which is larger than all subsolutions of \eqref{eq:Phi_inf}, and $\Phi_\infty = \bar\Phi$.
\label{lm:maxsonvisc}
\end{lemma}
Lemma \ref{lm:maxsonvisc} implies that any other solution $\underline{\Phi}$ of \eqref{eq:Phi_inf} satisfies $\underline{\Phi}\leq \Phi_\infty$, then obviously $\underline{\Phi} \leq \Phi(T,\cdot)$, but by the uniqueness result below $\Phi(T,\cdot)$ obtained in Step 3, we deduce that $\underline{\Phi}= \Phi_\infty$.
To achieve the proof of Theorem \ref{th2}, it now only remains to prove the previous lemma.

\begin{proof}[Proof of Lemma \ref{lm:maxsonvisc}]
  The proof follows closely the proof of Lemma \ref{le7} in Section \ref{se33}.
  For the sake of completeness, we repeat the argument in this framework.
Firstly, we build a sequence $\Phi_k(x)$ departing from $\Phi_0$ and defined by induction by, for $k\geq 0$, $\Phi_{k+1}(x)$ is a solution to 
  \[
  \left\{ \begin{array}{l}
    -\Delta \Phi_{k+1} + \gamma \Phi_{k+1} = S_0 \exp \left(\int_\Omega \beta(x,y)(\Phi_k(y)-\Phi_0(y))\,dy\right),
    \\[5pt]
     \partial_{n}\Phi_{k+1}=0 \quad \hbox{ on } \quad  \partial\Omega.
     \end{array} \right.
 \]
 Clearly, from \eqref{eq:Phi_0} and the maximum principle, we deduce that $\Phi_1 \leq \Phi_0$ on $\Omega$. Then iterating we deduce that $\Phi_{k+1} \leq \Phi_k$ for any $k\geq 0$. Hence $(\Phi_k)_k$ is a non-increasing sequence which converges to a solution of \eqref{eq:Phi_inf} denoted by $\bar\Phi$.
 
 Secondly, if $\underline{\Phi}$ is a subsolution of \eqref{eq:Phi_inf} which verifies $\underline{\Phi}\leq \Phi_0$. We have
 \begin{align*}
   -\Delta \underline{\Phi} + \gamma \underline{\Phi}\ 
   & \leq S_0 \exp \left(\int_\Omega \beta(x,y)(\underline{\Phi}(y)-\Phi_0(y))\,dy\right)  \\
   & \leq S_0 = -\Delta \Phi_{1} + \gamma \Phi_{1}  .
 \end{align*}
 By the maximum principle, we deduce that $\underline{\Phi}\leq \Phi_1$. Iterating, we deduce that $\underline{\Phi}\leq \Phi_k$ for any $k\geq 0$. Thus $\underline{\Phi}\leq \bar\Phi$. In particular, we have $\Phi_\infty\leq \bar\Phi$.

 Thirdly, we prove that $\bar\Phi\leq \Phi_\infty$. We have $\bar\Phi\leq \Phi_0$. Let us show that $\bar\Phi\leq \Phi(t)$ for all $t\geq 0$. By contradiction, if it is not true, there exists a first time $t_0$ for which $\Phi(t_0)$ touches $\bar\Phi$, i.e.\ $\bar\Phi\leq \Phi(t_0)$ and for some $x_0\in\Omega$, $\bar\Phi(x_0)=\Phi(t_0,x_0)$. Then, $\Delta \bar\Phi(x_0)\leq \Delta\Phi(t_0,x_0)$ and we have
 \begin{align*}
   S(t_0,x_0) + E(t_0,x_0) + I(t_0,x_0)\ 
   & = -\Delta \Phi(t_0,x_0) + \gamma \Phi(t_0,x_0) \\
   & \leq -\Delta \bar\Phi(x_0) + \gamma \bar\Phi(x_0) \\
   &  =  S_0 \exp \left(\int_\Omega \beta(x_0,y)(\bar\Phi(y)-\Phi_0(y))\,dy\right).
 \end{align*}
 Moreover, from \eqref{eq:Lyap}, we deduce
 \begin{align*}
   S(t_0,x_0)
   & = S_0 \exp \left(\int_\Omega \beta(x_0,y)\big(\Phi(t_0,y)-\Phi_0(y)\big)\,dy\right)  \\
   & \geq S_0 \exp \left(\int_\Omega \beta(x_0,y)\big(\bar\Phi(y)-\Phi_0(y)\big)\,dy\right).
 \end{align*}
 It is a contradiction which concludes the proof of Lemma~\ref{lm:maxsonvisc} and thus of uniqueness of $\Phi_\infty$.
\end{proof}

\section{Examples: finite rank contact matrices}
\label{se5}

The results of Section~\ref{sec:AMR} apply to finite rank contact matrices which have been studied widely with several applications as distributed susceptibility,\cite{thieme_book,thiemesisd}.
We show in Section \ref{se51} that for rank-1 contact matrices, the general formulas previously found may be simplified.
An alternative approach is presented in Section \ref{se52} in the case where $\gamma$ is constant, and generalized to finite rank matrices in Section \ref{se53}.

\subsection{SEIR system with rank-$1$ matrices}
\label{se51}

The case of rank-$1$ contact matrices refers to  matrices  $\beta(x,y)$ which we may be split into 
\[
\beta(x,y)=\beta (x) p(y).
\]
Then system \eqref{eqG1} reads
\begin{subequations}
\label{eq1}
\begin{gather} 
  \dot S(t,x) = -\beta(x) S(t,x) \int_\Omega p(y) I(t,y)\ dy,\qquad S(0,x) = S_0(x),  \label{eq1a} \\
  \dot E(t,x) = \beta(x) S(t,x) \int_\Omega p(y) I(t,y)\ dy - \alpha(x) E(t,x), \qquad E(0,x) = E_0(x), \label{eq1b} \\
\dot I(t,x) = \alpha(x) E(t,x) - \gamma(x) I(t,x),\qquad I(0,x) = I_0(x), \label{eq1c} \\
\dot R(t,x) = \gamma(x) I(t,x), \qquad R(0,x) = R_0(x), \label{eq1d}
\end{gather}
\end{subequations}
where $\alpha$ and $\gamma$ satisfy assumptions~\eqref{eq:hyp},  $\beta$ and $p$ are bounded and positive, and the initial data are given non-negative functions satisfying the conditions in~\eqref{eq:init}.
The result in Theorem \ref{th1} applies.

We can identify explicitly the next-generation operator $K_{S(t,\cdot)}$ in \eqref{def:Kt} as
\begin{align*}
  & K_{S(t,\cdot)} [\phi](x) = \beta(x) S(t,x) \int_\Omega \frac{p(y)}{\gamma(y)}\phi(y)\,dy,  \\
  & K^{\star}_{S(t,\cdot)}[ \psi](y) =  \frac{p(y)}{\gamma(y)} \int_\Omega \beta(x) S(t,x)\psi(x)\,dx.
\end{align*}
It is worth noticing that, for this operator, the principal eigenelements  are (up to normalisation) 
\[ 
\left\{ \begin{array}{l}
    r(K_{S(t,\cdot)}) = \int_\Omega\frac{\beta(y) p(y)}{\gamma(y)}S(t,y)\,dy,
    \\[5pt]
  \varphi_{S(t,\cdot)}(x)= \beta(x) \, S(t,x), \qquad   \psi_{S(t,\cdot)}(y)=  \frac{p(y)}{\gamma(y)}.
\end{array} \right.
\]
Indeed, both  $\varphi_{S(t,\cdot)}$ and $\psi_{S(t,\cdot)}$  are positive and satisfy the eigenvalue property
\[
K_{S(t,\cdot)} [\beta S(t,\cdot)](x) =  r(K_{S(t,\cdot)})\, \beta(x) S(t,x), \quad K^{\star}_{S(t,\cdot)} \Big[ \frac{p(\cdot)}{\gamma(\cdot)}\Big](y) =  r(K_{S(t,\cdot)})\,  \frac{p(y)}{\gamma(y)} .
\]

Hence, in this case, the basic reproduction number introduced in Definition~\ref{def:R0} just reads:
\begin{equation}
\label{eq555}
    \calR_0 := \int_\Omega S_0(x) \frac{\beta(x)p(x)}{\gamma(x)}\,dx,
\end{equation}
and the herd immunity domain is reached for distributions $\bar S$ such that 
\[
    \int_\Omega \bar S(x)\frac{\beta(x)p(x)}{\gamma(x)}\,dx < 1.
\]
Lastly, the computation of the final size reduces to 
  \begin{equation*}
    S_\infty (x)    = S_0(x) \exp\left(
      \beta(x) \int_\Omega \dfrac{p(y)}{\gamma(y)}\left(
       S_\infty(y) - S_0(y) - E_0(y) - I_0(y) \right)\,dy  \right),
\end{equation*}
with $S_\infty \leq S_0$.

\subsection{An alternative approach for SIR system with rank-$1$ matrices}
\label{se52}

In the simpler case of the SIR system with finite rank contact matrices, another method can be used to study the final size. We begin with rank-1 matrices, for which the formulas are simpler, and assume that $\gamma(x)$ is constant.
Analogously to \eqref{eq1}, the system reads
\begin{subequations}
\label{eq1SIR}
\begin{gather}
\label{eq1SIRa}
  \dot S(t,x) = -\beta(x) S(t,x) \int_\Omega p(y) I(t,y)\,dy,\qquad S(0,x) = S_0(x),  \\
\label{eq1SIRb}
  \dot I(t,x) = \beta(x) S(t,x) \int_\Omega p(y) I(t,y)\,dy- \gamma I(t,x) ,\qquad I(0,x) = I_0(x),\\
\label{eq1SIRc}
\dot R(t,x) = \gamma I(t,x), \qquad R(0,x) = R_0(x).
\end{gather}
\end{subequations}

In that case, one can reduce the system to a single ODE by introducing the quantity
\begin{equation}  \label{SIReqm}
m(t) :=\int_0^t \int_{\Omega} p(y) I(s,y) \,dy ds.
\end{equation}
As a matter of fact, the equation for $S(t,x)$ can be solved as
\[
\ln S(t,x) = \ln S_0(x) -  \beta(x) m(t) .
\]
Inserting expression \eqref{SIReqm} in the equation \eqref{eq1SIRb} for $I$,  and integrating against the weight $p(\cdot)$, we find
\[ \begin{array}{rl}
\displaystyle \frac{d}{dt}   \dot m (t) &=\displaystyle  \frac{d}{dt} \int_\Omega p(x) I(t, x) \,dx = \dot  m(t) \int_\Omega \beta(x) p(x) S(t,x) \,dx - \gamma \dot m(t)
\\
&\displaystyle =  \dot m(t) \int_\Omega \beta(x) p(x) S_0(x) e^{-\beta(x) m(t)}\,dx- \gamma \dot m(t)  .
\end{array}\]
Integrating once, one finds, since $\dot m(t=0) = \int_\Omega p(y) I_0(y)\,dy $, 
\begin{equation}
\dot m (t) = \int_\Omega p(x) [I_0+ S_0](x) \,dx  - \int_\Omega p(x) S_0(x) e^{-\beta(x) m(t)} \,dx- \gamma m(t), \qquad m(0)=0  .
\label{IntRk1}
\end{equation}

Herd immunity is reached when $\dot m=0$, that is when 
\begin{equation*}
\int_\Omega p(x) [I_0+ S_0](x) \,dx  =  \int_\Omega p(x) S_0(x) e^{-\beta(x) m_{H}} \,dx + \gamma m_{H}
=:F(m_H)
\end{equation*}
which, for any $I_0>0$  has a single root $m_H >0$ which is attractive. 
Indeed, $F$ is a smooth and convex function on $\mathbb{R}^+$ such that $F(0)< \int_\Omega p(x) [I_0+ S_0](x)\,dx$ and $\lim_{m\to +\infty} F(m)=+\infty$.
For $I_0$ small, $m_H$ may be close or far from $S_0$ depending on the stability of $I\equiv 0$ which is determined by $\calR_0= 1-\frac{F'(0)}{\gamma}$.

\subsection{SIR system with rank-$N$ matrices}
\label{se53}

The reduction to a single equation for rank-1 matrices can be extended to rank-$N$ contact matrices, i.e., which may be written
$$
\beta (x,y) = \sum_{i=1}^{N} \beta_{i}(x) p_{i}(y).
$$ 
Following the above construction, we introduce the matrix $M$ defined by 
\[
  M_{i,j}:= \int_\Omega \frac{\beta_{i}(x) p_{j}(x)}{\gamma(x)} S_{0}(x) \,dx.
\]
This matrix has positive coefficients and thus we can define its principal eigenvalue $\lambda$ and the associated eigenvector $X= (X_i)_{i\in\{1,\dots,N\}}$, with $X_i>0$ for all $i=1,...,N$.
\begin{lemma}
The basic reproduction number, defined in Definition~\ref{def:R0} is given by 
$$
\calR_0= r_{S_0(\cdot)}=\lambda
$$
 and (up to normalisation) the corresponding eigenfunctions are 
$$
\phi_{S_0(\cdot)} (x):=\sum_{i=1}^{N} X_{i}\beta_{i}(x) S_{0}(x), \qquad \psi_{S_0(\cdot)} (y):=\sum_{i=1}^{N} X_{i} \frac{p_{i}(y)}{ \gamma(y)}.
$$
\end{lemma}
\
\begin{proof} For the given expression of $\phi_{S_0(\cdot)}$, we compute
\begin{align*}
K_{S_0(\cdot)}[\phi_{S_0(\cdot)}] (x) =\ & \sum_{i,j=1}^{N} S_{0}(x) \beta_{i}(x)\int_y \frac{p_{i}(y)}{\gamma(y)} \beta_{j}(y)X_{j}\,dy \\
=\ &  \sum_{i,j=1}^{N} S_{0}(x) \beta_{i}(x) M_{i,j}X_{j} \\
=\ & \lambda \sum_{i=1}^{N} S_{0}(x) \beta_{i}(x) X_{i} = \lambda \phi_{S_0(\cdot)}
\end{align*}
which proves the result for $\phi_{S_0(\cdot)} $, the same calculation can be performed for $ \psi_{S_0(\cdot)} $.
\end{proof}

Notice that the characterization  of  $\calR_0$ given in the statement extends \eqref{eq555}.
\bigskip

For rank-$N$ matrices, it is also noticeable that one can also reduce the SIR system to $N$ differential equations as we did in the rank-1 case.
We  assume that $\gamma$ is independent of $x$ and define for $j=1,...,N$, 
\[
m_j (t) = \int_0^t p_j(y) I(t,y) dy, \qquad Q(t,x)= \sum_{k=1}^N \beta_k(x) m_k(t).
\]
We have
\[
\dot S(t,x) = - S(t,x) \dot Q(t,x), \qquad S(t,x) = S_0(x) e^{-Q(t,x)}.
\]

From the equation for infected, we compute
\begin{align*}
  \frac{d}{dt} \dot m_i(t) & = \int_\Omega p_i(x) S(t,x) \dot Q(t,x) dx - \gamma \dot m_i(t)  \\
  & = \int_\Omega p_i(x) S_0(x) e^{-Q(t,x)} \dot Q(t,x) dx- \gamma \dot m_i(t) .
\end{align*}
Integrating once in $t$, we obtain the equations, for $i=1,...,N$, 
\begin{equation}
\label{eq556}
 \dot m_i(t)  = - \int_\Omega p_i(x) S_0(x) e^{-Q(t,x) } dx - \gamma m_i(t) +  \int_\Omega p_i(x) [S_0(x)+I_0(x)]dx, \quad m_i(0)=0 .
 \end{equation}
Notice that \eqref{eq556} extends \eqref{IntRk1}.
Conceptually, this is a system of $N$ differential nonlinear equations, which can be computed at least numerically.

\section{Conclusion}

We  considered an epidemic in a heterogeneous population modelled by a SEIR system with a continuous structure variable and a general contact matrix. We  investigated question of the final size of the epidemic both with and without diffusion. In this general framework, our main contribution is to prove uniqueness results for the equation characterizing the final size of the epidemic. The proof combines a contraction property close to the steady state with a monotonicity argument which localized roughly the final state.

Although we presented our results for a SEIR model, they can be extended to other compartmental models, like the SIR model and to more general equations than the simple diffusion. 

The main limitation comes from the specific form of the SEIR model excluding birth and loss of immunity. In those models, up to our knowledge, proving merely convergence to a steady state is an open problem.

\appendix

\section{Proof of Lemma \ref{lem:Kdelta}}
\label{sec:appKdelta}

In this technical part, we show that the operator defined on $L^2(\Omega)$ by
\[
  K^{\Delta}_{S(\cdot)}[\varphi](x):=S(x) \int_{\Omega} \frac{\beta(x,y)}{\gamma(y)} \varphi(y)\,dy +\Delta \left(\frac{\varphi(x)}{\gamma (x)}\right),
\]
admits a principal eigenvalue, that is, that there exist $\lambda>0$ and a function $\varphi>0$ such that $ K^{\Delta}_{S(\cdot)}\varphi=\lambda \varphi$ in $\Omega$, with $\partial_{n}\varphi = 0$ in $\partial \Omega$.

It is more convenient to work with $\phi = \varphi / \gamma$. Let the operator $R$ be defined as
\[
  R\phi(x):=\frac{S(t,x)}{ \gamma (x)}\int_{\Omega}\beta(x,y)\phi(y)\,dy.
\]

We define for all $M>0$ the operator:
\[
  T\phi :=\Big( (-\Delta +M)^{-1} \circ R\Big) \phi,
\]
where $(-\Delta +M)^{-1} f$ is the unique solution $u$  of $-\Delta u+Mu = f$ in $\Omega$, with $\partial_{n}u=0$ on $\partial\Omega$.

This operator is clearly compact and strongly positive, from $L^{2}(\Omega)$ to  $L^{2}(\Omega)$. The Krein-Rutman theorem thus guarantees the existence of a principal eigenvalue $\Lambda_{M}$, associated with a positive eigenfunction $\phi_{M}$ that we normalize by $\int_{\Omega}(\phi_{M}(x))^{2}\,dx=1$. The identity $T\phi_{M}=\Lambda_{M}\phi_{M}$ rewrites
\begin{equation}\label{eq:eigenb}
  -\Delta \phi_{M}+M\phi_{M}= \frac{1}{\Lambda_{M}} R\phi_{M} \quad \hbox{in } \Omega.
\end{equation}
Moreover, the function $M\mapsto \Lambda_{M}$ is clearly continuous by Kato's regularity theory.

On the one hand, multiplying by $\phi_{M}$ and integrating by parts, one finds
\begin{align*}
  \frac{1}{\Lambda_{M}} \langle R\phi_{M},\phi_{M} \rangle_{L^{2}(\Omega)}&= \int_{\Omega}|\nabla \phi_{M} |^{2}\,dx + M\int_{\Omega}\big(\phi_{M}(x)\big)^{2}\,dx  \\
  &\geq M\int_{\Omega} \big(\phi_{M}(x)\big)^{2} \,dx.
\end{align*}
On the other hand, one has
\begin{align*}
  \langle R\phi_{M},\phi_{M} \rangle_{L^{2}(\Omega)} & = \int_{\Omega\times \Omega} \frac{1}{ \gamma (x)}\beta(x,y)S(t,x)\phi_{M}(x)\phi_{M}(y)\,dydx    \\
&\leq  C\left( \int_{\Omega}\phi_{M}(x)\,dx \right)^{2},
\end{align*}
for some constant $C>0$ due to (\ref{eq:hyp}) and $S(t,x) \leq S_{0}(x)$.
Hence, using the Cauchy-Schwarz inequality, one gets (increasing $C$ if necessary) $\frac{C}{\Lambda_{M}} \geq M$ and thus $\Lambda_{M}\to 0$ as $M\to +\infty$.

Moreover, integrating equation (\ref{eq:eigenb}) and using again our hypotheses  (\ref{eq:hyp}), one gets
\[
  M\int_{\Omega}\phi_{M}(x)\,dx = \frac{1}{\Lambda_{M}} \int_{\Omega\times \Omega} \frac{1}{ \gamma (x)}\beta(x,y)S(t,x)\phi_{M}(y)\,dydx \geq \frac{1}{C\Lambda_{M}}\int_{\Omega}\phi_{M}(x)\,dx
\]
for some positive constant $C>0$.
Hence, $CM \geq 1/\Lambda_{M}$ and thus  $\Lambda_{M}\to +\infty$ as $M\to 0$.

We conclude that there exists $M>0$ such that $\Lambda_{M}=1$ by the intermediate value theorem, and thus (\ref{eq:eigenb}) exactly yields, by letting $\varphi = \gamma \phi_{M}$:
\[
  K^{\Delta}_{S(\cdot)}\varphi = M\varphi,
\]
with $\varphi>0$ in $\Omega$. Letting $\lambda:=M$ ends the proof.

\section*{Acknowledgment}

The authors acknowledge the unknown referee for his/her comments and suggestions which allow to improve this paper.

B.P. has received funding from the European Research Council (ERC) under the European Union's Horizon 2020 research and innovation programme (grant agreement No 740623).

\end{color}


\end{document}